\documentclass[onefignum,onetabnum]{siamart171218}
\usepackage[margin=1.23in]{geometry}
\usepackage{appendix}
\usepackage{bbm}
\usepackage{todonotes}

\def\proof{\noindent{\it Proof}.}
\def\endproof{\ensuremath{\hfill\quad \square}\newline}
\makeatletter
\def\namedlabel#1#2{\begingroup
	#2%
	\def\@currentlabel{#2}%
	\phantomsection\label{#1}\endgroup
}
\makeatother

\newcommand{\R}{\mathbb{R}}

\newcommand{\N}{\mathbb{N}}
\newcommand{\tto}{\rightrightarrows}
\newcommand{\gph}{\operatorname{gph}}
\newcommand{\dom}{\operatorname{dom}}
\newcommand{\CC}{\mathcal{C}}
\newcommand{\ov}{\overline}
 \newcommand{\epi}{\operatorname{epi}}
 \newcommand{\cco}{\overline{\operatorname{co}}}
\def\Hat{\widehat}

\usepackage{lipsum}
\usepackage{amsfonts}
\usepackage{graphicx}
\usepackage{epstopdf}
\usepackage{algorithmic}
\usepackage{cite}
\usepackage{dsfont}
\usepackage{enumitem}
\usepackage{dsfont}
\usepackage{xcolor}
\usepackage{lmodern}
\usepackage{mathrsfs}
\usepackage{fix-cm}
\usepackage{amsmath}
\usepackage{amssymb}
\usepackage{latexsym}
\usepackage{bbm}
\usepackage{todonotes}

 \usepackage{hyperref}
\ifpdf
  \DeclareGraphicsExtensions{.eps,.pdf,.png,.jpg}
\else
  \DeclareGraphicsExtensions{.eps}
\fi

\renewcommand{\R}{\mathbb{R}}

\renewcommand{\N}{\mathbb{N}}
\renewcommand{\tto}{\rightrightarrows}
\renewcommand{\gph}{\operatorname{gph}}
\renewcommand{\dom}{\operatorname{dom}}
\renewcommand{\CC}{\mathcal{C}}
\renewcommand{\epsilon}{\varepsilon}
\renewcommand{\ov}{\overline}
 \renewcommand{\epi}{\operatorname{epi}}
 \renewcommand{\cco}{\overline{\operatorname{co}}}
\def\Hat{\widehat}

\newsiamremark{remark}{Remark}
\newsiamremark{example}{Example}
\newsiamremark{hypothesis}{Hypothesis}
\crefname{hypothesis}{Hypothesis}{Hypotheses}
\newsiamthm{claim}{Claim}

\headers{A Newton-Like Dynamical System for Nonsmooth Optimization}{J.G. Garrido, P. Pérez-Aros, and E. Vilches}

\title{A Newton-Like Dynamical System for Nonsmooth and Nonconvex Optimization\thanks{Submitted to the editors DATE.
\funding{Research of the first author was partially supported by ANID BECAS/DOCTORADO NACIONAL 21230802. Research of the second and third author was partially supported   by  Centro de Modelamiento Matem\'atico
(CMM), ACE210010 and FB210005, BASAL
funds for center of excellence and  ANID-Chile grant: MATH-AMSUD 23-MATH-09 and
 MATH-AMSUD 23-MATH-17, ECOS-ANID ECOS320027, Fondecyt Regular 1220886,
 Fondecyt Regular 1240335, Fondecyt Regular 1240120 and Exploración  13220097.}}}

\author{Juan Guillermo Garrido\thanks{Departamento de Ingenier\'ia Matem\'atica,  Universidad de Chile, Santiago, Chile
  (\email{jgarrido@dim.uchile.cl} ).}
\and Pedro Pérez-Aros\thanks{Departamento de Ingeniería Matem\'atica and Centro de Modelamiento Matem\'atico  (CNRS UMI 2807), Santiago, Chile 
  (\email{pperez@dim.uchile.cl}).}
\and Emilio Vilches\thanks{Instituto de Ciencias de la Ingenier\'ia, Universidad de O'Higgins, Rancagua, Chile (\email{emilio.vilches@uoh.cl}).}}

\usepackage{amsopn}

\makeatletter
\newcommand*{\addFileDependency}[1]{
  \typeout{(#1)}
  \@addtofilelist{#1}
  \IfFileExists{#1}{}{\typeout{No file #1.}}
}
\makeatother

\newcommand*{\myexternaldocument}[1]{%
    \externaldocument{#1}%
    \addFileDependency{#1.tex}%
    \addFileDependency{#1.aux}%
}

\ifpdf
\hypersetup{
  pdftitle={A Newton-Like Dynamical System for Nonsmooth and Nonconvex Optimization},
  pdfauthor={D. Doe, P. T. Frank, and J. E. Smith}
}
\fi

\myexternaldocument{ex_supplement}

\begin{document}

\maketitle

\begin{abstract}
  This work investigates a dynamical system functioning as a nonsmooth adaptation of the continuous Newton method, aimed at minimizing the sum of a  primal lower-regular and  a locally Lipschitz function, both potentially nonsmooth. The classical Newton method’s second-order information is extended by incorporating the graphical derivative of a locally Lipschitz mapping. Specifically, we analyze the existence and uniqueness of solutions, along with the asymptotic behavior of the system's trajectories. Conditions for convergence and respective convergence rates are established under two distinct scenarios: strong metric subregularity and satisfaction of the Kurdyka-Łojasiewicz inequality.
\end{abstract}



\def\hhat#1{\Hat{\vbox{\baselineskip=0pt\vskip-1.8pt
\hbox{$\Hat{\vbox{\baselineskip=0pt\vskip-
0.8pt\hbox{$\!#1$}}}$}}}}
\def\HHat#1{\widehat{\vbox{\baselineskip=0pt\vskip-1.5pt
\hbox{$\widehat{\vbox{\baselineskip=0pt\vskip-
0.5pt\hbox{$#1$}}}$}}}}
	

\setlength {\marginparwidth }{2cm} 

\def\proof{\noindent{\it Proof}.}
\def\endproof{\ensuremath{\hfill\quad \square}\newline}
%
%

%
%
%
%
%
\makeatletter
\def\namedlabel#1#2{\begingroup
	#2%
	\def\@currentlabel{#2}%
	\phantomsection\label{#1}\endgroup
}
\makeatother

\section{Introduction}

Newton's method is one of the most prominent algorithms for solving unconstrained optimization problems. By utilizing first- and second-order information, it achieves quadratic convergence rates, provided that the initial condition is sufficiently close to a critical point. The principles behind this algorithm have served as the foundation for the development of numerous efficient algorithms for both linear and nonlinear optimization. 

However, the requirements for the applicability of this algorithm can be quite restrictive in modern optimization problems, where the lack of smoothness or convexity of the objective function occurs frequently. From a more classical perspective, to address this difficulty, algorithms have been developed where, at each iteration, the objective function is replaced by a local approximation with the desired properties of differentiability or convexity. Notable examples of this paradigm include trust-region algorithms and difference-of-convex  methods (see, e.g., \cite{MR3785670}). 

A different paradigm involves directly addressing the non-differentiability or non-convexity of the objective function without relying on surrogates, using tools from generalized differentiation and variational analysis. In this direction, we highlight the development of non-smooth Newton methods, which, instead of requiring the continuous second-order differentiability of the function, only impose that the objective function belongs to the class $\mathcal{C}^{1,+}$ (see, e.g., \cite{MR4252735, MR4675935, MR4578210, MR4728882} and the references therein). Furthermore, in \cite{2023arXiv230103491A}, the authors introduce a novel approach for objective functions that can be decomposed as the difference between a $\mathcal{C}^{1,+}$ function and a prox-regular function. These algorithms, inspired by the classical Newton's method, leverage second-order information about the objective function using a class of generalized Hessians. Advanced techniques in variational analysis enable satisfactory convergence results for the proposed algorithms, along with faster convergence rates in certain cases compared to methods that rely solely on first-order information.

In this paper, we follow this second approach and we study the convergence of a Newton-like dynamical system associated with the optimization problem for objective functions that are not necessarily convex or differentiable.

The study of continuous versions of optimization algorithms is a classical topic in this field. It is widely recognized that continuous versions of discrete algorithms offer several advantages, both from a theoretical and practical perspective (see, e.g., \cite{MR3555044}). On the one hand, they enable convergence analysis through the stability theory of dynamical systems, and on the other hand, they facilitate the understanding of essential aspects of the algorithms, such as their derivation, and parameter tuning, stability improvements,  among others. In particular, different approaches to discretizing a continuous dynamical system lead to iterative algorithms, whose convergence analysis often parallels that of the continuous case.

In the continuous setting, the Newton method has been an interesting topic of study. 
Its global convergence was studied in \cite{MR0411577}, and in \cite{MR1624673}, the authors addressed the convergence rates in the convex and smooth setting. Subsequently, several works have explored inertial dynamics associated with the continuous Newton method (see, e.g., \cite{MR3985196, MR4328493, MR1930878, MR4592864, MR4695766}). 

Following the continuous setting, we propose a dynamical system associated with a nonsmooth Newton method for the minimization of nonsmooth and nonconvex functions. Formally, we consider the structured nonsmooth mathematical program 
    \begin{equation}\label{mainProblem}
         \min \varphi (x), \text{ with } \varphi:=\varphi_1 +\varphi_2,
    \end{equation}
    where $\varphi_1\colon \R^d\to \R\cup \{\infty\}$ is a primal lower regular (plr) function, and  $\varphi_2 \colon \R^d \to \R$ is locally Lipschitz function. It is important to highlight that the aforementioned optimization problem is quite general. On the one hand, working with extended-valued functions enables dealing with constrained optimization. On the other hand, plr functions cover a wide class of functions with an underlying nonsmooth structure. Among these functions, we can find all lower semicontinuous (lsc) proper convex functions, lower-$C^2$ functions, and qualified convexly composite functions (see, e.g., \cite{MR4659163}). 

With the problem \eqref{mainProblem} in mind, we propose the following dynamical system to model a fully nonsmooth continuous dynamical Newton-like method. We aim to find an absolutely continuous function \( x \colon \R_+ \to \mathbb{R}^d \) that satisfies the following conditions: There are measurable functions $v_1,v_2\colon \R_+\to \R^d$ such that  $v_1(t)\in \partial\varphi_1(x(t))$ and $v_2(t)\in \partial \varphi_2(x(t))$ for a.e. $t\in \R_+$ and
	\begin{equation*}
		\begin{aligned}
			v_1(t) + v_2(t)&\in -DF(x(t))(\dot{x}(t)) \text{ a.e. on } \R_+ \\
			x(0) &= x_0.
		\end{aligned}
	\end{equation*}
	where $x_0\in \R^d$ is given, and $F\colon \R^d\to \R^d$ is a locally Lipschitz function. Here, $\partial f $ denotes the Clarke subdifferential of a function $f$, and $DF$ represents the graphical derivative of the mapping $F$ (see definitions below). The dynamical system \eqref{general_system} is referred to as Newton-like because if $F$ is the gradient of a function, the system becomes the Newton method. Moreover, considering a general operator $F$ allows for the treatment of more general cases in the spirit of quasi-Newton method or Levenberg-Marquardt algorithm (see, e.g., \cite{MR3289054} and the references therein). 
 
The aim of this paper is to investigate the well-posedness and the asymptotic convergence of the solutions to the dynamical system \eqref{general_system}. 

 The paper is organized as follows: Section 2 introduces the mathematical background required for the subsequent analysis. Section 3 presents the main result concerning the existence and uniqueness of solutions for the system \eqref{general_system}. In Section 4, we study the asymptotic behavior of the solutions to the system \eqref{general_system}. In particular, we describe the limit set of the trajectories of \eqref{general_system} and demonstrate their convergence under the strong metric subregularity and the Kurdyka-{\L}ojasiewicz inequality. Rates of convergence are established in both cases. To keep it simple, all proofs and technical aspects of the paper are provided in the appendix.

	\section{Notation and preliminary results}
In what follows, we consider the standard concepts and definitions used in nonsmooth and variational analysis. For more details, we refer to the classical monographs \cite{MR1491362, MR3823783, MR4659163}. Throughout the paper, $\R^d$ denotes the $d$-dimensional Euclidean space with the norm $\|\cdot\|$ and the corresponding inner product $\langle\cdot,\cdot\rangle$. For $r > 0$ and $x \in \mathbb{R}^d$, we denote by $\mathbb{B}_r(x)$ and  $\mathbb{B}_r[x]$ the open and closed balls centered at $x$ with radius $r$, respectively. The indicator function of $S$ at $x$ is denoted by $\delta(x,S)$, which equals to $0$ if $x\in S$ and $\infty$ otherwise. the distance from a point $x$ to a nonempty set $S$ is defined as $\operatorname{dist}(x;S):=\inf_{y\in S}\Vert x-y\Vert$. 

\subsection{Variational geometry and generalized differentiation}

	Given $A\subset \R^d$ and $x\in A$, we say that $v$ belongs to the \emph{(Bouligand) tangent cone}, denoted by  $T_A(x)$, if there exist a sequence of positive numbers $(t_n)$ and $(v_n)\subset \R^d$ with $t_n\to 0$ and $v_n\to v$ such that  $x+t_nv_n\in A$ for all $n\in\N$.\\
The  \emph{Fr\'echet normal cone} to $A$ at  $x$ is defined by
\begin{equation*}
    \Hat{N}_A(x) := \left\{    v\in \mathbb{R}^d : \langle v, w  \rangle \leq 0, \text{ for all } w \in T_A(x)        \right\},
\end{equation*}
and, the \emph{Clarke normal} cone to $A$ at $x$, is defined as
\begin{equation*}
    N_A(x) :=\cco \left\{  v: \exists v_n, x_n\in A \textrm{ such that }v_n\in \Hat{N}_A(x_n) \textrm{ and } (x_n,v_n) \to (x,v) \right\},
\end{equation*}
where $\cco$ stands for the closed convex hull of a set.\\
 \noindent Let $f\colon \R^d\to\R\cup\{\infty\}$ be  a lower semicontinuous (lsc) function we recall that its epigraph is defined as $\epi f := \{ (x,\alpha)\in \R^d\times\R : f(x) \leq \alpha \}$. Moreover, for $x\in \R^n$,   we say that an element $\zeta$ belongs to the Clarke subdifferential of $f$ at $x$, denoted by $\partial f(x)$, if $( \xi, -1) \in {N}_{\epi f}(x, f(x))$.\\
A function $f\colon \subset \R^d\to \R\cup\{\infty\}$ is said to be weakly convex around $\bar x$ if there exists  a (convex) neighborhood   $U$ of $\bar x$    such that for all $x,y\in U$ and $\lambda\in [0,1]$ 
\begin{equation*}
    f(\lambda x + (1-\lambda)y)\leq \lambda f(x) + (1-\lambda)f(y) + \sigma\frac{\lambda(1-\lambda)}{2}\|x-y\|^2.    
\end{equation*}
\subsection{Primal Lower Regular Functions}
A notable class of functions is the so-called primal lower regular (plr) functions, originally introduced and studied under the name primal lower nice in \cite{MR1123210}, where their properties were developed in the finite-dimensional setting. This class constitutes an important subclass of prox-regular functions, which are now quite prominent in optimization (see, e.g., \cite{MR1333397}).  The results initially established in finite dimensions were later generalized to the infinite-dimensional setting in \cite{MR3108443}. The following definition of plr functions is taken from \cite[Definition 11.1]{MR4659163}, which is equivalent to the original definition (see \cite[Proposition 11.3]{MR4659163}). A function $f\colon \R^d \rightarrow  \R\cup\{\infty\}$ is said to be \emph{primal lower regular} (plr) at $\bar x$ if there are $\varepsilon >0$ and $c\geq 0$ such that $f$ is lsc on $\mathbb{B}_{\varepsilon}(\bar x)$ and for all $x \in \mathbb{B}_{\varepsilon}(\bar x) \cap \operatorname{dom} \partial f$ and all $\zeta \in \partial f(x)$, the following inequality holds: 
\begin{equation*}
	f(y) \geq f(x)+\langle\zeta, y-x\rangle-c(1+\|\zeta\|)\|y-x\|^2 \textrm{ for all }y \in \mathbb{B}_{\varepsilon}(\bar x).
\end{equation*}
The function $f$ is said plr on an open convex set $\mathcal{O}$ if it is plr at every point of $\mathcal{O}$. Moreover, if $f$ is plr on $\mathcal{O}$ with the same constant $c\geq 0$, we say that  $f$ is $c$-plr on $\mathcal{O}$. 

Let us emphasize that the class of plr functions is quite broad, encompassing $\mathcal{C}^{1,+}$ functions, weakly convex functions, lower semicontinuous (lsc) functions, and qualified convexly
composite functions (see, e.g., \cite[Theorem 11.13]{MR4659163}). Finally, according to \cite[Theorem 11.16]{MR4659163}, if \( f \) is plr at \( \bar{x} \), then several notions of subdifferentials coincide with the Clarke subdifferential at \( \bar{x} \).

\subsection{Tools from set-valued analysis}
	Given a set-valued mapping $F\colon \R^d\tto \R^d$, we define its domain as $\dom F := \{x\in\R^d : F(x)\neq\emptyset\}$ and its graph as $\gph F := \{(x,y)\in\R^d\times\R^d : y\in F(x)\}$. The graphical derivative of $F$ at $(\bar x,\bar y)\in \gph F$ in the direction $\bar v$ is defined by
	\begin{equation*}
		DF(\bar x,\bar y)(\bar v):= \{d\in \R^d:(\bar v,d)\in T_{\gph F}(\bar x,\bar y)\}.
	\end{equation*}
	For simplicity, when $F$ is single-valued, $DF(\bar x,\bar y)$ is denoted by $DF(\bar x)(\bar v)$. 

Let $G\colon \R^d\tto \R^d$ and let $\rho \in \mathbb{R}$. According to \cite{2023arXiv230103491A}, we say that $G$ is $\rho$-lower-definite if
 \begin{equation*}
\langle x,u\rangle\geq \rho \|x\|^2 \quad \text{ for all } \quad  (x,u)\in\gph G.
 \end{equation*}
 It is straightforward to observe that the above definition naturally extends the concept of positive definiteness of matrices. Furthermore, it can be demonstrated that   generalized derivatives of locally monotone operators are indeed lower-definite (see, for instance, \cite[Chapter 5]{MR3823783} for results related to  coderivative mapping and Lemma \ref{lem-smm-ld} for the graphical derivative).\\
        We say that a set valued mapping $F$ is \emph{strongly metrically subregular} at $(\bar{x},\bar{y})\in\gph F$ with modulus   $\kappa\geq 0$  if there are 
 if there are neighborhoods $U$ of $\bar{x}$ and $V$ of $\bar{y}$ such that $\|x-\bar{x}\|\leq \kappa \cdot \text{dist}(\bar y;F(x)\cap V), \forall x\in U$. It is known that the latter definition is equivalent to the existence of a neighborhood $U$ of $\bar x$ s.t. $\|x-\bar x\|\leq \kappa\cdot \text{dist}(\bar y; F(x)), \forall x\in U$ (see, e.g., \cite[p. 194]{MR3288139}).

	\section{Well-posedness of Newton-like Dynamical Systems}

In this section, we study the existence and uniqueness of solutions for the Newton-like dynamical system:
\begin{equation}\label{general_system}
\left\{
\begin{aligned}
&0\in \partial \varphi_1(x(t))+ \partial \varphi_2(x(t))+DF(x(t))(\dot{x}(t)) \quad \textrm{ for a.e. } t\in \mathbb{R}_+,\\
&x(0)=x_0.
\end{aligned}
\right.
\end{equation}
\begin{definition}
    We say that $x\colon \mathbb{R}_+\to \mathbb{R}^d$ is a \emph{solution} of the Newton-like dynamical system \eqref{general_system} if the following conditions hold:
    \begin{enumerate}
        \item[$(a)$] The map $t\mapsto x(t)$ is absolutely continuous on any compact interval of $\mathbb{R}_+$.
        \item[$(b)$] There exist $v,w\in L^2_{\operatorname{loc}}(\mathbb{R}_+;\mathbb{R}^d)$ such that 
        $$
        v(t)\in \partial \varphi_1(x(t)) \textrm{ and } w(t)\in \partial \varphi_1(x(t)) \textrm{ for a.e. } t\in \mathbb{R}_+.
        $$
        \item[$(c)$] For a.e. $t\in \mathbb{R}_+$,  $v(t)+w(t)\in -DF(x(t))(\dot{x}(t))$.
        \end{enumerate}

\end{definition}
 
Due to its generality, the system \eqref{general_system} may exhibit solutions that are not necessarily informative from an optimization point of view. Thus, we introduce the notion of an \emph{energetic solution}, which encompasses the regularity, integrability, and energy-decreasing properties of the trajectories, making it well-suited for optimizing the nonsmooth and nonconvex function $\varphi:=\varphi_1+\varphi_2$ via the dynamical system \eqref{general_system}.
\begin{definition}
  We say that a solution $x(\cdot)$ of \eqref{general_system} is \emph{energetic} if there exists $\rho>0$ such that the energy $t\mapsto E_{\rho}(t):=\rho \int_0^t \Vert \dot{x}(\tau)\Vert^2 ds+\varphi(x(t))$ is nonincreasing, that is, 
       \begin{equation}\label{integral_prop}
           \rho \int_{s}^t \Vert \dot{x}(\tau)\Vert^2 d\tau \leq \varphi (x(s))-\varphi(x(t)) \textrm{ for all } 0\leq s<t.
       \end{equation}
        To further emphasize the role of the energy $E_{\rho}(t)$, whenever this condition holds with a constant $\rho>0$, we will say that $x(\cdot)$ is a $\rho$-energetic solution. 
\end{definition}

\noindent \textbf{Standing Assumptions:} Recall that we are interested in optimizing the function $\varphi$, defined as the sum of two functions $\varphi_1\colon \mathbb{R}^d \rightarrow \mathbb{R} \cup \{\infty\}$ and $\varphi_2\colon \mathbb{R}^d \rightarrow \mathbb{R}$. Let $x_0 \in \operatorname{dom} \varphi$, and suppose there exists an open set $\Omega \subset \mathbb{R}^d$ such that the following standing assumptions hold for $\varphi_1$, $\varphi_2$, and $F\colon \mathbb{R}^d \rightarrow \mathbb{R}^d$:
	\begin{enumerate}
		\item[\namedlabel{FixH1}{$(\mathcal{H}_1)$}]  The level set $\{ x\in \R^d : \varphi(x)\leq \varphi(x_0) \}$ is contained within $\Omega$.
		\item [\namedlabel{FixH2}{$(\mathcal{H}_2)$}] The function $\varphi_1$ is plr on $\Omega$, and $\varphi_2$ is locally Lipschitz on $\Omega$.
        \item [\namedlabel{FixH3}{$(\mathcal{H}_3)$}] The function $F$ is locally Lipschitz on  $\Omega$.
	\end{enumerate}

We now state our main result concerning the existence of energetic solutions to the dynamical system \eqref{general_system}. Its proof is provided in Appendix \ref{sec-ap-B}.
 \begin{theorem}\label{main_theorem_prox_reg} 
 Let $x_0\in \dom \varphi$, and assume that  \ref{FixH1}, \ref{FixH2}, and \ref{FixH3} hold. In addition, assume that the following two conditions are also satisfied:
 \begin{enumerate}
     \item [\namedlabel{FixH4}{$(a)$}] The graphical derivative of $F$ is $\rho$-lower-definite on $\Omega$ for some $\rho>0$.
        \item[\namedlabel{FixH5}{$(b)$}] The function $\varphi$ satisfies a linear growth bound from below on $\Omega$, i.e., there is $\alpha>0$ such that $\varphi(x)\geq -\alpha(\|x\|+1)$ for all $x\in \Omega$.
 \end{enumerate}
  Then, there exists a $\rho$-energetic solution $x(\cdot)$ of the Newton-like dynamical system \eqref{general_system}, where $\rho$ is the constant given in \ref{FixH4}. Moreover, the map $t\mapsto \varphi(x(t))$ is absolutely continuous on any compact interval of $\mathbb{R}_+$. Finally, the above solution is unique if one of the following conditions holds:
 \begin{enumerate}
     \item[\namedlabel{Uniq1}{$(i)$}] $F$ is linear, and $\varphi_2=0$. 
     \item[\namedlabel{Uniq2}{$(ii)$}] The function $\varphi_1\in \mathcal{C}^{1,+}$ on $\Omega$, and $\varphi_2=0$.
 \end{enumerate}
	\end{theorem}
The uniqueness of solutions for \eqref{general_system} is a delicate matter that depends on the monotonicity properties of the system. The following example illustrates the nonuniqueness of solutions for the Newton-like dynamical system \eqref{general_system}. 

\begin{example} Consider $\Omega = \R$, $F(x) = x$, $\varphi_1\equiv 0$, $\varphi_2(x) = \min\{ |x-1|,|x+1| \}$, and $x_0 = 0$. It is clear that all the assumptions of Theorem \ref{main_theorem_prox_reg} are satisfied. However, both $x_1(t)=-\min\{t,1\}$ and $x_2(t)=\min\{t,1\}$ solve \eqref{general_system}.
\end{example}
\section{Asymptotic Analysis of Trajectories}
In this section, we study the asymptotic behavior of the trajectories of the Newton-like dynamical system \eqref{general_system}. We demonstrate that the set of accumulation points of any bounded trajectory of the Newton-like dynamical system is nonempty, compact, and connected. Furthermore, if the trajectory is energetic, then its accumulation points are stationary in a suitable sense for the optimization problem \eqref{mainProblem}. We then establish the convergence of the trajectories, providing a quantified rate of convergence under conditions of strong metric regularity and the Kurdyka-{\L}ojasiewicz inequality. 

To study the optimization problem \eqref{mainProblem}, we consider the following notion of a stationary point.
\begin{definition}
A point $x^{\ast}\in \mathbb{R}^d$ will be called a \emph{stationary point} of the optimization problem \eqref{mainProblem} if $ 0\in\partial\varphi_1(x^\ast) + \partial\varphi_2(x^\ast)$. The set of stationary points will be denoted by $\mathcal{S}$.
\end{definition}

Given a trajectory $x\colon \mathbb{R}_+\to \R^d$, the $\omega$-\emph{limit} of $x(\cdot)$, also called the  \emph{set of accumulation points} of $x(\cdot)$, is defined as  
$$
\omega(x):=\{z\in \mathbb{R}^d : \exists (t_n)\subset \mathbb{R}_+ \textrm{ with } t_n\to \infty \textrm{ and } x(t_n)\to z \textrm{ as } n\to \infty\}.
$$
The next result relates the set of accumulation points of the Newton-like dynamical system \eqref{general_system} to the stationary points of the minimization problem \eqref{mainProblem}. The proof is given in Appendix \ref{theorem_asym_behavior}.
 \begin{theorem}\label{asym_behavior} Assume that \ref{FixH1}, \ref{FixH2} and \ref{FixH3} hold. Let $x(\cdot)$ be a  solution of the Newton-like dynamical system \eqref{general_system}. Then the following hold:
		\begin{enumerate}
			\item [$(a)$] If $x(\cdot)$ is bounded, then the set $\omega(x)$ is nonempty, compact, and connected.
		\item [$(b)$] If $x(\cdot)$ is an energetic solution, then $\omega(x)\subset \mathcal{S}$. Moreover, if $x^\ast\in \omega(x)$, there exists a sequence $s_n \to \infty $ such that $x(s_n) \to x^\ast$, $\dot{x}(s_n) \to 0$, $\varphi(x(s_n)) \to \varphi(x^\ast)$, and 
   $$
 \operatorname{dist}(0;\partial \varphi_1 (x(s_n)) + \partial \varphi_2 (x(s_n))) \to 0 \textrm{ as } n\to \infty.
   $$
  In addition, if $x^{\ast}$ is an isolated point of $\omega(x)$, then  $x(t)\to x^{\ast}$ as $t\to \infty$.
		\end{enumerate}
	\end{theorem}

 \subsection{Convergence of Energetic Solutions Under Strong Metric Subregularity}

In this section, we establish the convergence of energetic solutions of the Newton-like dynamical system \eqref{general_system} under the strong metric subregularity of  the subdifferential of the objective function $\varphi$. The proof is presented in Appendix \ref{Proofofrateundermetric_reg}.

\begin{theorem}\label{rateundermetric_reg}
Assume that \ref{FixH1}, \ref{FixH2} and \ref{FixH3} hold. Let $x(\cdot)$ be an energetic solution of \eqref{general_system},  and suppose that $x^{\ast}\in \omega(x)$ is such that the mapping $\partial\varphi_1+\partial\varphi_2$ is strongly metrically subregular at $(x^{\ast},0)$ with modulus $\kappa>0$. Then $x(t)\to x^{\ast}$ as $t\to \infty$, and there exist $\sigma, T, \nu>0$ such that 
\begin{align}
    \sigma	\int_s^t \| x(\tau) - x^\ast\|^2\,d\tau &\leq 	\varphi(x(s))-\varphi(x(t)) & \textrm{ for all } T\leq s\leq t,\label{eq02}\\
		\min_{s\in [T,t]} \| x(s) - x^\ast\| &\leq \frac{\nu}{\sqrt{t-T}} & \textrm{ for all } t>T.  \label{eq_trivial}
    \end{align}
If, in addition, one of the following conditions holds:
 \begin{enumerate}
     \item [$(i)$] The function $\varphi_1$ is weakly convex around $x^\ast$ and $\varphi_2=0$,
     \item [$(ii)$] $\varphi_1 = 0$ and $-\varphi_2 $ is weakly convex around $x^\ast$,
 \end{enumerate}
then there exist $\alpha, \beta, \gamma,\tau>0$ such that for all $t\geq \tau$, we have 
	\begin{equation}\label{eq01}
			\gamma \int_t^{\infty}   \operatorname{dist}^2(0;\partial \varphi(x(s)))\, ds \leq 	 \varphi(x(t))-\varphi(x^\ast)  \leq \alpha \exp(-\beta t).
	\end{equation}
 Moreover, in the case $(i)$, if the function $\varphi_1$ is $\sigma$-weakly convex with $\sigma<\kappa^{-1}$, then there exist $\hat \alpha,\hat  \beta, \hat \tau>0$ such that
 \begin{equation}\label{eqn-bound-varphi1}
     \|x(t)-x^\ast\|\leq \hat \alpha \exp(-\hat \beta t) \quad \text{ for all } t\geq \hat \tau.
 \end{equation}

\end{theorem}

\begin{remark}
     It is worth noting that the exponential rate of convergence established in Theorem \ref{rateundermetric_reg} is indeed tight. For any strongly convex $\varphi\in \mathcal{C}^{1,+}$ function, the unique solution of \eqref{general_system} with $F(x) = \nabla \varphi(x)$ and any initial point $x_0$ satisfies
\[
\| \nabla \varphi(x(t)) \|^2 = e^{-2t} \| \nabla \varphi(x_0) \|^2.
\]
Thus, for all $t\geq 0$, one has
\[
\int_t^\infty \| \nabla \varphi(x(s)) \|^2 \, ds = \frac{1}{2} e^{-2t} \| \nabla \varphi(x_0) \|^2,
\]
which shows that \eqref{eq01} is tight. \\
Furthermore, consider the particular case where $\varphi(x) = \frac{1}{2} \| x \|^2$ and $F(x) = \nabla \varphi(x) = x$. In this case, the unique solution of  \eqref{general_system} is given by $x(t) = e^{-t} x_0$, which clearly converges to the solution of the optimization problem \eqref{mainProblem}. Moreover, it satisfies $\| x(t) - x^\ast \| = e^{-t} \| x_0 \|$, demonstrating that \eqref{eqn-bound-varphi1} is also tight. 
 \end{remark}
 
\subsection{Convergence of Solutions Under the Kurdyka-{\L}ojasiewicz Inequality}
Another theoretical framework for studying the convergence of dynamical systems in a quantified manner, which is based on hypotheses regarding the functional-analytic properties of functions, is the so-called Kurdyka-{\L}ojasiewicz property (see, e.g., \cite{MR2274510, MR3341671}). We say a function $\varphi$ satisfies the \emph{structured Kurdyka-{\L}ojasiewicz property} at $x^\ast$ if there exist $\eta>0$ and a continuous concave function $\psi\colon [0, \eta] \rightarrow[0, \infty)$ with $\psi(0)=0$ such that $\psi$ is $\mathcal{C}^1$-smooth on $(0, \eta)$ with a strictly positive derivative $\psi^{\prime}$, satisfying 
	\begin{equation}\label{KL}
		\psi^{\prime}(\varphi(x)-\varphi(x^\ast)) \operatorname{dist}(0 ; \partial \varphi_1(x)+\partial\varphi_2(x)) \geq 1
	\end{equation}
	for all $x \in \mathbb{B}_\eta(x^\ast)$ with $\varphi(x^\ast)<\varphi(x)<\varphi(x^\ast)+\eta$. 
    
    The following result establishes the convergence of energetic solutions of the dynamical system \eqref{general_system} under this property, which proof can be found in Appendix \ref{ProofOf_thm-KL}.
	\begin{theorem}\label{thm-KL}
		Assume that \ref{FixH1}, \ref{FixH2}, and \ref{FixH3} hold. Let $x(\cdot)$ be an energetic solution of \eqref{general_system}, and suppose that $x^{\ast}\in \omega(x)$ is such that $\varphi$ satisfies the structured Kurdyka-{\L}ojasiewicz property at $x^{\ast}$. Then, $\dot{x}\in L^1(\R_+;\R^{d})$, $x(t) \to x^\ast$, and  $\varphi(x(t)) \to \varphi(x^\ast) $ as $t\to \infty$.
	\end{theorem}

The following corollary establishes the convergence rates for the solution $x(\cdot)$ of the system \eqref{general_system}, assuming that the function $\psi$ satisfies a specific form of the structured 
 Kurdyka-{\L}ojasiewicz inequality \eqref{KL}. The proof is given in Appendix \ref{Proof_of_Coro-KL}.
 \begin{corollary}\label{Coro-KL}
     In the setting of Theorem \ref{thm-KL}, suppose that $\psi(t) = Mt^{1-\theta}$, where $M>0$ and $\theta\in [0,1[$. Then, there exists $T>0$ such that:
     \begin{enumerate}
         \item [(a)] If $\theta\in [0,1/2[$, $x(t)$ and $\varphi(x(t))$ converge to $x^\ast $ and $\varphi(x^\ast)$, respectively, in finite time.
         \item [(b)] If $\theta = 1/2$, there exist $\alpha,\beta>0$ such that $\varphi(x(t))-\varphi(x^\ast)\leq \beta\exp(-\alpha t)$ and $\|x(t)-x^\ast\|\leq \beta\exp(-\alpha t/2)$  for all $t\geq T$.
         \item [(c)] If $\theta\in ]1/2,1[$, there exist $\mathscr{G}>0$ such that $\|x(t)-x^\ast\|\leq \mathscr{G}t^{\frac{1-2\theta}{1-\theta}}$ and $\varphi(x(t))-\varphi(x^\ast)\leq \mathscr{G} t^{\frac{1}{1-2\theta}}$ for all $t \geq T$.
     \end{enumerate}
 \end{corollary}
 
\appendix
\section{Technical results from variational analysis}
 
   The following result provides a characterization of plr functions in terms of hypomonotonicity (see, e.g., \cite[Theorem 25.7]{MR3108443}).

	\begin{theorem}
		Let $f\colon  \R^d \rightarrow \R\cup\{\infty\}$ be a function which is finite at $\bar{x} \in X$ and lsc near $\bar{x}$. Then $f$ is plr at $\bar{x}$ if and only if there exist $\varepsilon>0$ and $c\geq 0$ such that
			\begin{equation}\label{hypo_plr}
				\left\langle\zeta_1-\zeta_2, x_1-x_2\right\rangle \geq-c\left(1+\left\|\zeta_1\right\|+\left\|\zeta_2\right\|\right)\left\|x_1-x_2\right\|^2
			\end{equation}
			whenever $\zeta_i \in \partial f(x_i)$ and $\left\|x_i-\bar{x}\right\| \leq \varepsilon$ for $i=1,2$. 
	\end{theorem}
    The \emph{Moreau envelope} of parameter $\lambda>0$ of a proper function  $f\colon \R^d\to \R\cup\{\infty\}$ is defined by 
\begin{equation}\label{mor_env}
    e_\lambda f(x) := \inf_{z\in\R^d} (f(z) + \frac{1}{2\lambda}\|x-z\|^2).
\end{equation}
The set of points where infimum in \eqref{mor_env} is attained is denoted by $P_\lambda f(x)$.\\
   \noindent Now, we recall important properties concerning the Moreau envelope of a plr function (see \cite{MR3108443,MR2252239}), in particular, its local  differentiability.
 
 \begin{lemma}\label{envelope_lemma}
		Let us consider $h\colon \R^d\to\R\cup\{\infty\}$ an lsc function and $x_0\in\text{dom }h$. Assume $h$ is $c$-plr on an open convex set containing $\mathbb{B}_{s_0}[x_0]$ where $s_0 \in ]0, \frac{1}{2c}[$. Take any $(\bar x,\bar y)\in\gph\partial h$ with $\|\bar x- x_0\|<\frac{s_0}{18}$ and let
		\begin{equation*}
			\bar{h}(\cdot)=h(\cdot)+\delta\left(\cdot, \mathbb{B}_{s_0/2}\left[\bar x\right]\right)\text{, }  r_0 := \frac{s_0}{18}.
		\end{equation*}
		Then there exists some threshold $\lambda_0$ such that for any $\lambda \in]0, \lambda_0[$:
		\begin{enumerate}[label=(\alph*), ref=\ref{envelope_lemma}-(\alph*)]
			\item  
   $e_\lambda \bar{h}$ is $\CC^{1,+}$ on $\mathbb{B}_{r_0}(x_0)$.
			\item \label{env_b} $P_\lambda \bar{h}$ is nonempty, single-valued, and $k_0$-Lipschitz continuous on $\mathbb{B}_{r_0}(x_0)$ for
			$
			k_0:=1+\lambda_0+\frac{1}{1-2c(\lambda_0+s_0/2)}.
			$
			\item 
   $\nabla e_\lambda \bar{h}=\lambda^{-1}\left(I-P_\lambda \bar{h}\right)$ on $\mathbb{B}_{r_0}(x_0)$.
			\item 
   $P_\lambda \bar{h}(\mathbb{B}_{r_0}(x_0)) \subset \mathbb{B}_{(63/8)r_0}(x_0)\subset \mathbb{B}_{s_0/2}(x_0)$.
			\item \label{env_e}
   For all $x \in \mathbb{B}_{r_0}(x_0)$,  $\nabla e_\lambda \bar{h}(x) \in \partial h(P_\lambda \bar{h}(x))$.
			\item 
   $\nabla e_\lambda \bar{h}$ is $k_\lambda$-Lipschitz on $\mathbb{B}_{r_0}(x_0)$ for $k_\lambda =  \frac{1}{\lambda(1-2c(\lambda(1+\|\bar{y}\|)+s_0/2))}+\|\bar y\| $.  
			\item \label{env_g} Assume $x\in \mathbb{B}_{r_0}(x_0)$, then 
			\begin{equation*}
				\lim_{y\to x,\lambda\searrow 0}P_\lambda \bar{h}(y) =  x\text{, }\liminf_{y\to x, \lambda\searrow 0}e_\lambda\bar{h}(y)\geq h(x) \text{ and } \lim_{\lambda\searrow 0}e_\lambda\ov{h}(x) = h(x).
			\end{equation*}
		\end{enumerate}
	\end{lemma}
	
	\begin{proof}
		The statements \textit{(a)}, \textit{(b)}, \textit{(c)}, \textit{(d)} and \textit{(e)} are direct consequences of \cite[Proposition 2.8]{MR2252239}. The assertion \textit{(f)} is implied from \textit{(b)} and \textit{(c)}. Moreover, using \textit{(b)}, for $x\in \mathbb{B}_{r_0}(x_0)$ and $\lambda<\lambda_0$ we have 
		\begin{equation*}
			\bar{h}(x)\geq \bar{h}(P_\lambda\bar{h}(x)) + \frac{1}{2\lambda}\|x-P_\lambda\bar{h}(x)\|^2.
		\end{equation*}
		From \textit{(d)} it follows that $P_\lambda\bar{h}(x)\in \mathbb{B}_{s_0/2}[\bar x]$, thus
		\begin{equation*}
			h(x)\geq h(P_\lambda\bar{h}(x))+\frac{1}{2\lambda}\|x-P_\lambda\bar{h}(x)\|^2\geq \inf\{ h(z):z\in \mathbb{B}_{s_0}(x_0)\} + \frac{1}{2\lambda}\|x-P_\lambda\bar{h}(x)\|^2.
		\end{equation*}
		Then, $\|x-P_{\lambda}\bar{h}(x)\|\leq \sqrt{2\lambda(h(x)-\inf\{ h(z):z\in \mathbb{B}_{s_0}(x_0)\})}$. Now, take $y\in\mathbb{B}_{r_0}(x_0)$,  from  then using \textit{(b)} for $\lambda<\lambda_0$
		\begin{equation*}
				\|P_\lambda\bar{h}(y)-x\| \leq \|P_\lambda\bar{h}(y)-P_\lambda\bar{h}(x)\| + \|P_\lambda\bar{h}(x)-x\| \leq k_0\|y-x\| + \|P_\lambda\bar{h}(x)-x\|.
		\end{equation*}
		It follows that $P_\lambda\bar{h}(y)\to x$ when $y\to x$ and $\lambda\searrow 0$. We also note that for all $\lambda\in ]0,\lambda_0[$,
		\begin{equation*}
			\begin{aligned}
				&e_\lambda\bar{h}(y)= h(P_\lambda\bar{h}(y))+\frac{1}{2\lambda}\|y-P_\lambda\bar{h}(y)\|^2\geq h(P_\lambda\bar{h}(y))\\
				\implies & \liminf_{y\to x,\lambda\searrow 0} e_\lambda\bar{h}(y)\geq \liminf_{y\to x,\lambda\searrow 0} h(P_\lambda\bar{h}(y))\geq h(x),
			\end{aligned}
		\end{equation*}
		where we have used $h$ is lsc and the fact recently proved. Finally, noting that for all $x\in \mathbb{B}_{r_0}(x_0)$ and $\lambda\in ]0,\lambda_0[$, $e_\lambda\bar{h}(x)\leq h(x)$, we conclude that $\lim_{\lambda\searrow 0}e_\lambda\bar{h}(x) = h(x)$.
	\end{proof}
 
      We say that a set valued mapping $F\colon \R^d\tto\R^d$ is \textit{strongly monotone} around $(\bar{x}, \bar{v})$ if there exist $\rho >0$ and a neighborhood $U \times V$ of $(\bar{x}, \bar{v})$ such that for any pair $\left(u_1, v_1\right),\left(u_2, v_2\right) \in \operatorname{gph} F \cap(U \times V)$ we have the estimate
		\begin{equation}\label{eq_strong_mon}
			\left\langle v_1-v_2, u_1-u_2\right\rangle \geq \rho\left\|u_1-u_2\right\|^2 .
        \end{equation}
        We say that $F$ is locally strongly maximal monotone around $(\bar x,\bar y)$ if $F$ satisfies \eqref{eq_strong_mon} on some neighborhood $U\times V$ of $(\bar x,\bar y)$ and for every monotone mapping $S\colon \R^d\tto \R^d$ such that $\gph F\cap U\times V\subset \gph S$ one has $\gph F\cap U\times V = \gph S\cap (U\times V)$.
        
	For a set-valued mapping $F\colon \R^d\tto \R^d$,  we say that $F$ admits a \textit{single-valued localization} around $(\bar{x}, \bar{y})\in\gph F$ if there is a neighborhood $U \times V \subset \R^d\times \R^d$ of $(\bar{x}, \bar{y})$ such that the mapping $\widehat{F}\colon U \rightarrow V$ defined via  $\gph\widehat{F}:=\operatorname{gph} F \cap(U \times V)$ is single-valued on $U$ with $\dom \widehat{F}=U$. Furthermore, we say $F$ admits a \textit{Lipschitzian single-valued localization} around $(\bar{x}, \bar{y})$ if the mapping $\widehat{F}$ is Lipschitz continuous on $U$, equivalently, we can say $F^{-1}$ is \emph{strongly metrically regular} (see \cite[p. 185]{MR3288139}). 
 
  The next lemma shows that the graphical derivative of a strongly locally monotone operator is $\rho$-lower-definite, where $\rho$ is the  same modulus of strong local monotonicity. 
	\begin{lemma}\label{lem-smm-ld}
		Let $F\colon \R^d\tto \R^d$ a set-valued mapping such that $F$ is strongly locally monotone at $(\bar x , \bar y)\in\gph F$ with modulus $\rho>0$, then $DF(\bar x,\bar y)$ is $\rho$-lower-definite.
	\end{lemma}
	
	\begin{proof}
		We have there is a neighborhood $U\times V$ of $(\bar x , \bar y)$ such that \eqref{eq_strong_mon} is satisfied on $\gph F\cap (U\times V)$. Take any $d\in DF(\bar x , \bar y)(z)$ for $z\in \R^d$, then, $(z,d)\in T_{\gph F}(\bar x , \bar y)$, it implies that there are sequences $(s_k)\subset \R_+$ and $(z_k,d_k)$ such that $s_k\searrow 0$, $(z_k,d_k)\to (z,d)$ and $(\bar x , \bar y) + s_k(z_k,d_k)\in \gph F$. Then, for all $k\in\N$ big enough, we have 
		\begin{equation*}
			\langle (\bar y + s_k d_k) - \bar v,(\bar x + s_k z_k) - \bar x \rangle \geq \rho\|\bar x + s_k z_k - \bar x\|^2
		\end{equation*}
		it follows that $\langle d_k,z_k\rangle\geq \rho\|z_k\|^2$, then taking the limit when $k\to \infty$, we get $\langle d,z\rangle\geq \rho\|z\|^2$, so we conclude that $DF(\bar x,\bar y)$ is $\rho$-lower-definite.
	\end{proof} 

 The following result (see \cite[Theorem 5.3.2]{MR2458436}) is a chain rule for the composition of a vector-valued Lipschitz function and an absolutely continuous function, expressed using the graphical derivative.
\begin{lemma}\label{lemma_cotingent_lips}
    Let $F\colon \R^d\to \R^d$ be a map which is locally Lipschitz on the open set $U\subset \R^d$ and $x\colon \R_+\to \R^d$ be an absolutely continuous function such that $x(\R_+)\subset U$. Then, we have that for a.e. $t\in \R_+$  
    \begin{equation*}
        DF(x(t))(\dot{x}(t)) = \left\{ \frac{d}{dt}(F\circ x)(t) \right\}.    
    \end{equation*}
\end{lemma}

The following result shows that lower-definite and locally Lipschitz mappings are strongly metrically regular.
\begin{proposition}\label{prop-s-m-r}
     Let $F\colon \R^d\to \R^d$ be a map where $F$ is locally Lipschitz around $\bar x$ and $DF$ is $\rho$-lower-definite around $\bar x$. Then, $F$ is strongly metrically regular at $(\bar x,F(\bar x))$.
 \end{proposition}
 \begin{proof}
 Take $\delta>0$ such that $x\tto DF(x)$ is $\rho$-lower-definite for all $x\in\mathbb{B}_\delta(\bar x)$. Also, we suppose that $F$ is Lipschitz on $\mathbb{B}_\delta(\bar x)$. Take $x_1,x_2\in \mathbb{B}_\delta(\bar x)$. Define $\alpha\colon [0,1]\to \R^d$ given by $\alpha(t) = tx_1 + (1-t)x_2$. It follows from Lemma \ref{lemma_cotingent_lips}   that $DF(\alpha(t))(x_1-x_2) = \{ \frac{d}{dt}(F\circ \alpha)(t) \}$ a.e. on $[0,1]$. Then,
 \begin{equation*}
     \begin{aligned}
         \langle F(x_1)-F(x_2),x_1-x_2 \rangle &= \langle (F\circ\alpha)(1)-(F\circ\alpha)(0),x_1-x_2 \rangle\\
         &= \left\langle \int_0^1 \frac{d}{dt}(F\circ \alpha)(t)dt,x_1-x_2 \right\rangle\\
         &= \int_0^1 \left\langle \frac{d}{dt}(F\circ \alpha)(t)dt,x_1-x_2 \right\rangle dt\\
         &\geq\int_0^1 \rho\|x_1-x_2\|^2 dt = \rho \|x_1-x_2\|^2
     \end{aligned}
 \end{equation*}
 where we have used that $DF(\alpha(t))$ is $\rho$-lower-definite. The above inequality implies that $F$ is strongly locally monotone around $(\bar x,F(\bar x))$ with modulus $\rho$. Now, suppose that $S$ is a monotone operator such that $\gph F\cap (\mathbb{B}_{\delta}(\bar x)\times \R^d)\subset \gph S$. Then, let us consider $(a,b)\in \gph S$ with $a\in \mathbb{B}_\delta(\bar x)$, then we have that $a+\varepsilon (b-F(a))\in \mathbb{B}_\delta(\bar x)$ for all $\varepsilon >0$ small enough, then by virtue of the monotonicity of $S$ we have
 \begin{equation*}
        \langle F(a+\varepsilon (b-F(a)))-b, a+\varepsilon (b-F(a)) - a\rangle\geq 0
 \end{equation*}
 it follows that $\langle F(a+\varepsilon (b-F(a)))-b, b-F(a)\rangle\geq 0$ for all $\varepsilon >0$ small enough. Since $F$ is continuous on $\mathbb{B}_{\delta}(\bar x)$, taking $\varepsilon\searrow 0$ yields $\|F(a)-b\|^2\leq 0$ thus $F(a) = b$, so $\gph F\cap (\mathbb{B}_\delta(\bar x)\times \R^d) = \gph S\cap (\mathbb{B}_\delta(\bar x)\times \R^d)$, it means that $F$ is strongly locally maximal monotone, and by virtue of \cite[Theorem 5.13]{MR3823783} we finally conclude that $F^{-1}$ admits a Lipschitzian single valued localization.
 \end{proof}
 
The next result  is a chain rule for general lower semicontinuous functions and absolutely continuous functions in terms of regular subgradients.
Recall that an element $\zeta\in\R^d$ belongs to the \emph{Fr\'echet subdifferential} of $f$ at $\bar x$, denoted by $\Hat\partial f(\bar x)$, if $( \xi, -1) \in \Hat{N}_{\epi f}(\bar x, f(\bar x))$.
	\begin{lemma}\label{lemma_chain_rule}
		Let $\varphi\colon \R^d\to \R\cup\{\infty\}$ and $x\colon \R_+\to\R^d$ be functions such that $\varphi\circ x$ and $x$ are derivable at $t\in\R_+$. Then, for all $v\in \Hat\partial \varphi(x(t))$, one has
	$ 
			\frac{d}{dt}(\varphi\circ x)(t) = \langle v,\dot{x}(t)\rangle.
	$
	\end{lemma}
	\begin{proof}
		Consider  $\varepsilon>0$ arbitrary. By definition, there is $\delta>0$ such that for all $y\in \mathbb{B}_{\delta}(x(t))$, one has 
		\begin{equation*}
			\varphi(y)\geq \varphi(x(t)) +\langle v,y-x(t)\rangle - \varepsilon\|y-x(t)\|
		\end{equation*} 
		Take $y=x(t+h)$ where for all $h$ small enough, $x(t+h)\in \mathbb{B}_{\delta}(x(t))$, then, 
		\begin{equation}\label{eqn_frechet}
			\varphi(x(t+h))\geq \varphi(x(t)) +\langle v,x(t+h)-x(t)\rangle - \varepsilon\|x(t+h)-x(t)\|
		\end{equation}
		so dividing by $h>0$ and taking the limit when $h\searrow 0$ we obtain that $\frac{d}{dt}(\varphi\circ x)(t)\geq \langle v,\dot{x}(t)\rangle-\varepsilon\|\dot{x}(t)\|$. For other side, dividing by $h<0$ the inequality \eqref{eqn_frechet} and taking $h\nearrow 0$, we get that $\frac{d}{dt}(\varphi\circ x)(t)\leq \langle v,\dot{x}(t)\rangle+\varepsilon\|\dot{x}(t)\|$. Finally, taking $\varepsilon\searrow 0$ we conclude the desired equality.
	\end{proof}

Let us recall that the subdifferential $\partial f$ of a locally Lipschitz function $f\colon \mathbb{R}^d \to \mathbb{R}$ is \textit{conservative} if, for every absolutely continuous curve $\gamma\colon [0,1] \to \mathbb{R}^d$ with $\gamma(0) = \gamma(1)$, the following holds:
\begin{equation*}
    \int_{0}^{1}\max_{v \in \partial f(\gamma(t))}\langle \dot{\gamma}(t), v \rangle\, dt = 0.
\end{equation*}
We refer to  \cite{MR4276581} for more details about the previous definition.  

 We recall the equivalence between being conservative and the chain rule for locally Lipschitz functions. It can be found in \cite[Corollary 2]{MR4276581}.
\begin{proposition}
    Let $f\colon \R^d\to\R$ be a locally Lipschitz function, then the following assertions are equivalent:
    \begin{enumerate}
        \item [(a)] $\partial f$ is conservative.
        \item [(b)] For any absolutely continuous function $x\colon [0,1]\to \R^d$ one has 
        \begin{equation*}
            \frac{d}{dt}(f\circ x)(t) = \langle v,\dot{x}(t)\rangle, \text{ for all } v\in\partial f(x(t))\text{ and for a.e. } t\in [0,1].
        \end{equation*}
    \end{enumerate}
\end{proposition}
 The following lemma is an elementary property involving real numbers.
\begin{lemma}\label{numbers_lemma}
    Let $A,B,C\in \R_+$ and $x\in\R$ satisfying the following inequality: $Ax^2\leq Bx+C$ then one has $|x|\leq \frac{B}{A}+\sqrt{\frac{C}{A}}$.
\end{lemma}	
\vspace{-0.2cm}
   \section{Proof of existence and uniqueness}
   \subsection{Preparatory results}
   The next result, a consequence of Lemma \ref{lemma_cotingent_lips}, provides an integral representation of \eqref{general_system} under additional assumptions.
\begin{proposition}\label{equivalence_prop}
    Suppose that $F\colon \R^d\to \R^d$ is a locally Lipschitz function on $\Omega$. If $x\colon \R_+\to \R^d$ is an absolutely continuous function and $v_1,v_2\colon \R_+\to \R^d$ are measurable functions with $v_i(t)\in\partial\varphi(x_i(t))$ a.e.  for $i\in\{1,2\}$ and 
    \begin{equation}\label{general_system_int}
         F(x_0) - \int_0^t v_1(s)+v_2(s)ds = F(x(t)) \quad \textrm{ for all } t\geq 0.
    \end{equation}
    Then, $x$ satisfies \eqref{general_system}. Moreover, if $x\colon \R_+\to \R^d$ is a solution of \eqref{general_system} with $x(\R_+)\subset \Omega$, then there are measurable functions $v_1,v_2\colon \R_+\to\R^d$ which satisfies \eqref{general_system_int}.
\end{proposition}

Now, we establish sufficient conditions to ensure that any absolutely continuous solution $x\colon \R_+\to \R^d$ of \eqref{general_system} is, in fact, energetic.
 \begin{proposition}
 Assume that \ref{FixH1}, \ref{FixH2} and \ref{FixH3} hold and the graphical derivative of $F$ is $\rho$-lower-definite at every point of $\Omega$. Let us consider an absolutely continuous solution $x\colon \R_+\to \R^d$ of \eqref{general_system}. Suppose that  $\partial\varphi_2$ is conservative on $\Omega$, and one of the following conditions holds
     \begin{enumerate}
         \item [(a)] $\varphi_1$ is locally bounded from above on $\Omega$.
         \item [(b)] $\varphi\circ x$ is nonincreasing on $\R_+$.
     \end{enumerate}
     then $x(t)\in\Omega$ for all $t\in \R_+$ and it is an energetic solution. Furthermore, in case (a) the function $\varphi \circ x$ is absolutely continuous.
 \end{proposition}
\begin{proof}
    Take $T := \inf\{t\geq 0: x(t)\notin \Omega\}$, by continuity and $x_0\in \Omega$, we have $T>0$ and $x(t)\in \Omega$ for all $t\in [0,T[$. Take $v\colon \R_+\to \R^d$ a measurable function such that 
     \begin{equation*}
         v(t)\in -DF(x(t))(\dot{x}(t))\cap (\partial\varphi_1(x(t))+\partial\varphi_2(x(t))) \text{ a.e. on }\R_+,
     \end{equation*}
     which implies that
$$
\rho\|\dot{x}(t)\|^2\leq \langle -v(t),\dot{x}(t)\rangle \quad \textrm{ for a.e. } t\in [0,T[.
$$
First case: Assumption \textit{(a)} holds.\\
     According to \cite[Proposition 25.5]{MR3108443}, if $\varphi_1$ is locally bounded from above on $\Omega$ we have that $\varphi_1$  is locally Lipschitz  there. Since $\varphi$ is locally Lipschitz on $\Omega$ and $x$ is absolutely continuous, we have $\varphi\circ x$ is absolutely continuous. By virtue of Lemma \ref{lemma_chain_rule} and given that $\varphi_2$ is conservative, we have a.e. on $[0,T[$, $\frac{d}{dt}(\varphi\circ x)(t) = \langle v(t),\dot{x}(t)\rangle\leq -\rho\|\dot{x}(t)\|^2$ (since $F$ is $\rho$-lower-definite). It implies that for all $s,t\in [0,T[$:
     \begin{equation}\label{eqn_int_kp}
         \rho\int_s^t\|\dot{x}(\tau)\|^2d\tau\leq \varphi(x(s))-\varphi(x(t)).
     \end{equation}
     Therefore, \eqref{integral_prop} holds. Finally, suppose that $T<\infty$. According to \eqref{eqn_int_kp}, it turns out $\varphi(x(t))\leq \varphi(x_0)$ for all $t\in [0,T[$ and since $\varphi$ is lsc one has $\varphi(x(T))\leq \varphi(x_0)$. Then $x(T)\in\Omega$, which contradicts the definition of $T$.  Hence, $T=\infty$ and the conclusion holds when \textit{(a)} occurs.\\
        Second case: Assumption \textit{(b)} holds.\\
        Now, assume \textit{(b)} holds, we have directly $T=\infty$. Since $\varphi\circ x$ is nonincreasing it follows that it is derivable almost everywhere. Using that $\varphi_2$ is conservative and Lemma \ref{lemma_chain_rule} we have a.e.
        \begin{equation*}
            \rho\|\dot{x}(t)\|^2\leq \langle v(t),\dot{x}(t)\rangle=-\frac{d}{dt}(\varphi\circ x)(t).
        \end{equation*}
        Then, by integrating, for all $t,s\geq 0$:
        \begin{equation*}
            \rho\int_s^t\|\dot{x}(\tau)\|^2d\tau\leq \int_s^t -\frac{d}{d\tau}(\varphi\circ x)(\tau)d\tau\leq \varphi(x(s))-\varphi(x(t)),
        \end{equation*}
        which implies that \eqref{integral_prop} holds.
\end{proof}
  \begin{lemma}\label{lemma_local_solution}
 Let $F\colon \R^d\to \R^d$ be a map such that $F$ is locally Lipschitz around $x_0$ and $DF$ is $\rho$-lower-definite around $x_0$. Suppose $\varphi  \in \mathcal{C}^{1,+}$ around $x_0$. Then,  there are  $T>0$  and an unique absolutely continuous function $x\colon [0, T]\to \R^d$   such that $x(0)=x_0$ and 
 	\begin{equation}\label{eqn_integral_lemma}
 	      F(x_0) - \int_0^t \nabla  \varphi(x(t)) ds = F(x(t)), 	\forall t\in [0,T].
 	\end{equation}
  Moreover for a.e. $t\in [0,T]$, $-\nabla\varphi(x(t))= DF(x(t))(\dot{x}(t))$ and for all $s,t\in [0,T]$ with $s<t$ 
  \begin{equation}\label{eqn-descent01111}
      \rho\int_s^t\|\dot{x}(\tau)\|^2d\tau\leq \varphi(x(s))-\varphi(x(t)).
  \end{equation}
 \end{lemma}
 
 \begin{proof}
 	  By using Proposition \ref{prop-s-m-r}, $F$ is strongly metrically regular at $(x_0,F(x_0))$, there is a $\eta$-Lipschitz localization of $F^{-1}$ around of $(F(x_0),x_0)$ for some $\eta>0$, it means there is $\theta\colon V'\to U'$ where $V',U'$ are neighborhoods of $F(x_0)$ and $x_0$ with $\gph\theta = \gph F^{-1}\cap V'\times U'$. By shrinking $U'$, we can suppose that $DF$ is $\rho$-lower-definite on $U'$.\\
 	Now, let us consider $r_0>0$ such that $\mathbb{B}_{r_0}[x_0]\subset   U'$,  $\theta^{-1}( \mathbb{B}_{r_0}[x_0]  ) \subset   V'$,  $\mathbb{B}_{r_0} [F(x_0)] \subset  V'$, and such that $\varphi$ is $\mathcal{C}^{1,+}$ on $\mathbb{B}_{2r_0}(x_0)$. We set the constant
 	\begin{equation*}
 	    T := \min\{ r_0, \eta^{-1}\} \left(  \sup\left\{  \| \nabla \varphi(x)\| : x  \in \mathbb{B}_{r_0}[x_0]  \right\}+2L\right)^{-1},
 	\end{equation*}
 	where $L>0$ is the Lipschitz constant of $\nabla \varphi$   on $ \mathbb{B}_{r_0}[x_0] $. Using this constant we  define  the   (complete) metric space $(\mathcal{X},d_\infty)$ with 
  \begin{equation*}
      \begin{aligned}
          \mathcal{X}:= \{ x \in   \CC([0,T]; \mathbb{B}_{r_0}[x_0]): x(0)=x_0    \} \textrm{ and }d_{\infty} (x,y) := \sup_{t\in [0,T] } \| x(t) - y(t) \|  .
      \end{aligned}
  \end{equation*}
 	\noindent\textbf{Claim 1:} The operator $\Phi \colon \mathcal{X} \to \mathcal{X}$, given by
 		\begin{equation*}
 			\Phi\colon x\mapsto \Phi(x)(t):=\theta(F(x_0)-\int_{0}^{t} \nabla \varphi (x(s)) ds)
 		\end{equation*}  
 		is a (well-defined)  contraction, and consequently it has a unique fixed point. \\
 	\noindent \emph{Proof of Claim 1:} Indeed, consider any $x \in \mathcal{X}$ and   $   t\in [0,T]$. It follows that   
 	\begin{equation*}
 		\left\|  \int_{0}^{t} \nabla\varphi (x(t) ) ds \right\|  \leq T \sup\left\{  \| \nabla \varphi(x)\| : x  \in \mathbb{B}_{r_0}[x_0]  \right\}\leq r_0.
 	\end{equation*}
 	Therefore, $F(x_0)-\int_{0}^{t} \nabla \varphi (x(s)) ds$ belongs to $\mathbb{B}_{r_0}[F(x_0)]$ and consequently  
 	$\Phi_\lambda(x)(t)$ is well-defined for each $t \in [0, T]$, and a continuous function on $t$. Moreover, for any $x, y\in \mathcal{X}$, we have that 
 	\begin{equation*}
 		\begin{aligned}
 			\|\Phi(x)(t)-\Phi(y)(t)\| &\leq  \eta\int_0^t \|\nabla {\varphi}(x(s)) -\nabla {\varphi}(y(s))\| ds\\
 			 & \leq  L\cdot\eta \int_0^{T} \|x(s)-y(s)\|\leq L\cdot\eta\cdot T\cdot d_\infty (x,y) \leq \frac{1}{2}d_\infty (x,y).
 		\end{aligned}
 	\end{equation*}
  	The last part follows from the Banach Contraction Theorem. The claim is proved. \\	
\noindent\textbf{Claim 2:} The (unique) fixed point of the operator $\Phi$ satisfies  \eqref{eqn_integral_lemma}.  \\
 		\noindent \emph{Proof of Claim 2:} It follows that  for all $t \in [0,T]$,
 		\begin{equation*}
 		x(t) = 	\theta(F(x_0)-\int_{0}^{t} \nabla \varphi (x(s)) ds),
 		\end{equation*}
 which implies that $F(x_0)-\int_{0}^{t} \nabla \varphi (x(s)) ds = F(x(t))$. Hence, $x$  satisfies \eqref{eqn_integral_lemma}. Finally, by Lemma \ref{lemma_cotingent_lips}, for a.e. on $[0,T]$, we have $-\nabla\varphi(x(t))= DF(x(t))(\dot{x}(t))$. We also conclude \eqref{eqn-descent01111} by using that $DF$ is $\rho$-lower-definite on $U'$. Indeed, a.e. on $[0,T]$, $-\frac{d}{dt}(\varphi\circ x)(t)=\langle -\nabla\varphi(x(t)),\dot{x}(t) \rangle\geq \rho\|\dot{x}(t)\|^2$, then we integrate on $[s,t]$ for $s<t$ with $s,t\in [0,T]$.
 	\end{proof}

  The next proposition is a technical tool that allows us to extend solutions of generalized equation to maximal elements, and this will be used in the construction of the solution of the generalized system \eqref{general_system}.  In order to present this result as a general tool, for a given map $F \colon \mathbb{R}^d \to \mathbb{R}^d$ and $x_0\in\R^d$, lower semicontinuous functions $\varphi_1,\varphi_2 \colon \mathbb{R}^d \to \R\cup\{\infty\}$, and a nonempty set $U\subset \R^d$, we define $\mathcal{K}:=\mathcal{K}_{\varphi_1, \varphi_2, F,U}(x_0)$ as the set of all elements $(x,v,w, \tau)$ such that  $\tau \in (0, \infty]$,  $x\colon [0,\tau [\to U$  is an absolutely continuous function such that $x(0) = x_0$, $\varphi\circ x$ is absolutely continuous  and there are  functions  $v,w \in L_{\textnormal{loc}}^2([0,\tau[, \R^d)$ with $v(s)\in\partial \varphi_1(x(s))$ a.e. and $w(s)\in\partial \varphi_2(x(s))$ a.e. such that
 \begin{align}
  F(x(t)) &= F(x_0) - \int_0^t \left(v(s)+w(s)\right) ds, 	\text{ for all } t\in [0,\tau [,\label{eqn_integral_extension_lemma}\\
  	 \rho\int_s^t \|\dot{x}(\tau)\|^2 d\tau&\leq \varphi(x(s))-\varphi(x(t)), \text{ for all } t,s\in [0, \tau[ \text{ with } s< t. \label{integral_prop_lemma}
 \end{align}
 We order $\mathcal{K}$ by $\prec$ as follows: $(x_1,v_1,w_2,\tau_1)\prec (x_2,v_2,w_2,\tau_2)$ if and only if $\tau_1\leq \tau_2$ and
 \begin{equation*}
 	{x_2}(t) = x_1(t), \, v_2(t) = v_1(t) \text{ and } w_2(t)= w_1(t) \text{ for all } t\in [0, \tau_1[.
 \end{equation*}
 \begin{proposition}\label{Prop_Maximal}
 	The set $\mathcal{K}$ has a maximal element, provided that $\mathcal{K}$ is nonempty. 
 	\end{proposition}
 \begin{proof}
 	  Consider an arbitrary chain in $\mathcal{K}$, that is, a subset $\{(x_i,v_i,w_i,\tau_i):i\in I\}$ totally ordered, our goal is to prove it has a upper bound. Indeed, in fact we can define $\tau := \sup_{i\in I}\tau_i$ and $x\colon [0,\tau[\to {U}$ and $v,w \colon  [0,\tau[\to \mathbb{R}^d$  as
 	  \begin{equation*}
 	  	x(t) = x_i(t) \text{ if } t<\tau_i, \quad v(t) = v_i(t) \text{ if } t<\tau_i, \quad  w(t)	= w_i(t) \text{ if } t<\tau_i.
 	  \end{equation*}
 	It follows that they are well defined due to $\{(x_i,v_i,w_i,\tau_i)\}$ is a chain. Moreover, it is easy to see that $x,v,w$ satisfy   \eqref{eqn_integral_extension_lemma} and \eqref{integral_prop_lemma} on  $[0, \tau[$  since every $x_i$ satisfies this in the interval $[ 0, \tau_i[$. So, by using Zorn's lemma, $\mathcal{K}$ has a maximal element.
 	\end{proof}
\vspace{-0.3cm}

\subsection{Proof of Theorem \ref{main_theorem_prox_reg}}\label{sec-ap-B} 
		    Since $F$ satisfies \ref{FixH4}, by Lemma \ref{prop-s-m-r}, $F$ is strongly metrically regular at every point $(x,F(x))$ with $x\in \Omega$. Since  $\varphi_1$ is plr at $x_0$, there exists   constants $c,s_0>0$ such that the assumptions of Lemma \ref{envelope_lemma} holds. Particularly, $\varphi$ is $c$-plr on $\mathbb{B}_{s_0}(x_0)$. Moreover,  we can shrink $s_0>0$ such that $\mathbb{B}_{s_0}(x_0)\subset \Omega$, $\inf\{\varphi(x):x\in \mathbb{B}_{s_0}(x_0)\}>-\infty$, \eqref{hypo_plr} holds on $\mathbb{B}_{s_0}(x_0)$ for some $\hat{c}\geq 0$ and $\varphi_2$ is Lipschitz on $\mathbb{B}_{s_0}(x_0)$.  \\
				Then,  we apply  Lemma \ref{envelope_lemma} to $\varphi_1$ with any point $(\bar x, \bar y) \in \gph  \partial \varphi_1$ then it yields the existence of $\lambda_0$ such that the conclusions of Lemma \ref{envelope_lemma} holds. In what follows we  are going to use the notation introduced in the aforementioned result. Particularly, we represent by $e_\lambda \bar{\varphi}_1$ the Moreau envelope of 	
    \begin{equation*}
			\bar{\varphi }_1(\cdot)=\varphi_1(\cdot)+\delta\left(\cdot, \mathbb{B}_{s_0/2}\left[\bar x\right]\right)\text{, }  r_0 := \frac{s_0}{18}.
		\end{equation*}
				We also consider a smooth mollifier $\psi\colon\R^d\to\R$ such that $\text{supp }\psi\subset \mathbb{B}$ and then we define the function 
		\begin{equation*}
			\varphi_2^\lambda(x) := \int_{\R^d} \varphi_2(x-z)\cdot \frac{1}{\lambda^d}\psi\left(\frac{z}{\lambda}\right)dz = \int_{\R^d}\varphi_2(x-\lambda z)\psi(z)dz.
		\end{equation*}
		To simplify the notation,  we define the set $U:=\mathbb{B}_{r_0/2}[x_0]$ and  the function $\varphi_\lambda :=  e_\lambda\bar{\varphi}_1+\varphi_2^\lambda$.
If it is necessary, we can shrink $\lambda_0$ in order to have
\begin{equation}\label{eqn-grad-varphi2}
    \sup_{\lambda\in ]0,\lambda_0[}\sup_{x\in U}\|\nabla\varphi_2^\lambda(x)\|\leq L,
\end{equation}
for some $L>0$, which is possible since $\varphi_2$ is Lipschitz on $\mathbb{B}_{s_0}(x_0)$. Note that $\varphi_\lambda\in \mathcal{C}^{1,+}$ on $U$. Now, for each $\lambda \in ]0, \lambda_0[$, we can employ Lemma \ref{lemma_local_solution} and Proposition \ref{Prop_Maximal} to get an element  $(x_\lambda, v_\lambda,w_\lambda, \tau_\lambda) \in \mathcal{K}_\lambda:=\mathcal{K}_{e_\lambda \bar{\varphi}_1, \varphi_2^\lambda, F,U}(x_0)$ which is maximal. In this case, by the choice of $r_0$ we have that $ v_\lambda(t) = \nabla  e_\lambda\bar{\varphi}_1 (x_\lambda(t))$  and $ w_\lambda(t) = \nabla \varphi_2^\lambda (x(t))$, so $v_\lambda(t) + w_\lambda (t)= 	\nabla \varphi_\lambda (x_\lambda(t))$. Using \ref{FixH4}, we get that  for almost all $t\in [0,\tau_\lambda[$, 
		\begin{equation*}
			\begin{aligned}
				\rho\|\dot{x}_\lambda(t)\|^2 \leq \langle -\nabla e_{\lambda}\bar{\varphi}_1(x_\lambda(t))-\nabla \varphi_2^\lambda(x_\lambda(t)),\dot{x}_\lambda(t) \rangle
				\leq -\frac{d}{dt}(e_\lambda\bar{\varphi}_1(x_\lambda(t)))+L\|\dot{x}_\lambda(t)\|.
			\end{aligned}
		\end{equation*} 
		Integrating the last inequality, we obtain that for all $\eta<\tau_\lambda$ 
		\begin{equation*}
			\begin{aligned}
				\rho\int_0^\eta\|\dot{x}_\lambda(t)\|^2dt \leq& \ e_{\lambda}\bar{\varphi}_1(x_0)-e_{\lambda}\bar{\varphi}_1(x_\lambda(\eta)) + L\int_0^\eta\|\dot{x}_\lambda(t)\|dt\\
				\leq& \ \varphi_1(x_0) -\inf_{y\in\R^d} e_{\lambda}\bar{\varphi}_1(y) + L\sqrt{\eta}\left(\int_0^\eta\|\dot{x}_\lambda(t)\|^2dt\right)^{1/2}\\
				=& \ \varphi_1(x_0) - \inf_{ \mathbb{B}_{s_0}(x_0)    } \varphi_1+  L\sqrt{\eta}\left(\int_0^\eta\|\dot{x}_\lambda(t)\|^2dt\right)^{1/2},
			\end{aligned}
		\end{equation*}
	 where we have used Hölder's inequality.  From Lemma \ref{numbers_lemma} we conclude
	\begin{equation}\label{eq_bounded}
		\left(\int_0^\eta\|\dot{x}_\lambda(t)\|^2dt\right)^{1/2}\leq   \kappa(\eta), \text{ 	where $\kappa(\eta):= \frac{L\sqrt{\eta}}{\rho} + \sqrt{\frac{1}{\rho}\left(\varphi_1(x_0) -\inf_{ \mathbb{B}_{s_0}(x_0)    } \varphi_1 \right) } $}.
	\end{equation}
		
		\noindent\textbf{Claim 1:} $\liminf\limits_{ \lambda \to 0} \tau_\lambda>0$.\\
\noindent \emph{Proof of Claim 1:} Indeed, suppose by contradiction that $\liminf_{ \lambda\to 0} \tau_\lambda = 0$. Then, there exists a subsequence $\lambda_j \to 0$, such that $\tau_{\lambda_j} \to 0$. Then, we fix $j \in\mathbb{N}$, by using \eqref{eq_bounded}, we notice that  for all $s,t\in [0,\tau_{\lambda_j}[$ with $t\geq s$
		   \begin{equation}\label{equi}
		   		\|x_{\lambda_j}(t)-x_{\lambda_j}(s)\|\leq\int_s^t \|\dot{x}_{\lambda_j}(t')\|dt' \leq \sqrt{t-s}\left(\int_0^t\|\dot{x}_{\lambda_j}(t')\|^2dt'\right)^{1/2}\leq \kappa(\tau_{\lambda_j}) \sqrt{t-s} .
		   \end{equation}
		  Since $\varsigma :=\sup_{j\in\N}\kappa(\tau_{\lambda_j})<\infty$, there exists $\lim_{t\nearrow \tau_{\lambda_j}}x_{\lambda_j}(t)$. We define 
		  \begin{equation*}
		  	 \hat{x}_j := \lim_{t\nearrow \tau_{\lambda_j}}x_{\lambda_j}(t) \text{ and  }   \hat{y}_j := F(x_0) + \int_0^{\tau_{\lambda_j} } \nabla \varphi_{\lambda_j} (x_{\lambda_j}(s)) ds.
		  \end{equation*}
		For \eqref{equi}, for all $j\in\N$, we have $\|\hat{x}_j-x_0\|\leq\varsigma \sqrt{\tau_{\lambda_j}}$. Since $\tau_{\lambda_j}\to 0$, for $i$ big enough, $\hat{x}_i\in \mathbb{B}_{r_0/4}(x_0)$, and since $F$ is locally Lipschitz,  $\hat{y}_i = F( \hat{x}_i )$. Since, $\hat{x}_i  \in \Omega$, we have that $F$ is strongly metrically regular at $(\hat{x}_i, \hat{y}_i) $. Therefore, applying Lemma \ref{lemma_local_solution}, we can extend the function $x_{\lambda_i}$ into $U$, which contradicts the maximality of this element, and that concludes the proof of the claim. 
 
 \noindent \textbf{Claim 2:} There is $\hat{\lambda}\in ]0, \lambda_0[$, such that $(x_{\lambda})_{ \lambda < \hat \lambda}$ is equicontinuous  and pointwise bounded family of functions, with  $(\dot{x}_\lambda)_{ \lambda < \hat \lambda}$,  $(v_\lambda)_{ \lambda < \hat \lambda}$ and $(w_\lambda)_{ \lambda < \hat \lambda}$ bounded in $L^2([0, \beta], \R^d)$.\\
 \noindent \emph{Proof of Claim 2:} Let us consider $\hat{\lambda}>0$ and $\beta >0$ with $\inf_{\lambda < \hat{\lambda}} \tau_{\lambda}> \beta$. First, using \eqref{equi}, $(x_\lambda)$ is an equicontinuous and pointwise bounded family. Second,  let us notice that by \eqref{eq_bounded} the family of functions $\dot{x}_\lambda$, with $\lambda < \hat{\lambda}$ is bounded  in $L^2([0,\beta], \R^d)$. Moreover,  since $x_\lambda([0,\beta])\subset U$ and \ref{FixH3} holds, we can shrink even more $s_0$ in order to $F$ is Lipschitz on $U$. Then, one of the following conditions holds
 \begin{equation*}
 		\|\nabla e_\lambda \bar{\varphi}_1(x_\lambda(t))+\nabla \varphi_2^\lambda(x_\lambda(s))\|\leq  \Tilde{L} \|\dot{x}_\lambda(t)\| \text{ a.e. on }[0,\beta]
 \end{equation*}
 for some $\Tilde{L}>0$. Therefore, joint with \eqref{eqn-grad-varphi2}, we can assume that the family of measurable functions $v_\lambda(t) $ and $w_\lambda(t)$ are bounded in $L ^2 ([0,\beta], \R^d  )$.\\
 	\noindent\textbf{Claim 3:} There are subsequences $\lambda_n \to 0$,  an absolutely continuous function $x\colon [0,\beta[ \to \Omega$ and $v , w \in L_{\text{loc}}^2([0,\beta], \R^d)$ such that $x_{\lambda_n} \to x$ uniformly, and $\dot{x}_{\lambda_n}\rightharpoonup  u$, $v_{\lambda_n}\rightharpoonup  v$ and $w_{\lambda_n} \rightharpoonup w$ in $L^2([0,\beta];\R^d)$.\\
  \noindent \emph{Proof of Claim 3:} By virtue of Arzela-Ascoli's Theorem, it has a subsequence converging uniformly to a continuous function $x\colon [0,\beta]\to \R^d$. Furthermore, by boundedness of $(\dot{x}_\lambda)$, $(v_\lambda)$ and $(w_\lambda)$ in $L^2([0, \beta], \R^d)$ we  can extract weakly convergent subsequences. Thus, we consider $x_{\lambda_n} \to x$  uniformly $[0,\beta]$ and $\dot{x}_{\lambda_n}\rightharpoonup  u$, $v_{\lambda_n}\rightharpoonup  v$ and $w_{\lambda_n} \rightharpoonup w$ in $L^2([0,\beta];\R^d)$.\\
 	Finally,  since $x_\lambda$ is absolutely continuous for all $\lambda\in ]0,\hat{\lambda}[$,
 	\begin{equation*}
 		\forall t,s\in [0,\tau], \forall n\in\N: x_{\lambda_n}(t)-x_{\lambda_n}(s) = \int_s^t \dot{x}_{\lambda_n}(\eta) d\eta.
 	\end{equation*}
 	By taking the limit when $n\to\infty$ it yields $\forall t,s\in [0,\beta]: x(t)-x(s) = \int_s^t u(\eta) d\eta$, the above equality implies that $x$ is absolutely continuous and $u = \dot{x}$. 
 	
 	\noindent\textbf{Claim 4:} The functions  $x$, $v $ and $ w$ satisfy  \eqref{eqn_integral_extension_lemma}. \\
\noindent  	\emph{Proof of Claim 4:} On the other hand,     for all $t\in [0,\beta] $ and all $ n\in\N$ we have 
		\begin{equation*}
			 F(x_{\lambda_n}(t)) = F(x_0)-\int_{0}^t  \left( v_{\lambda_n} (s) + w_{\lambda_n} \right(s))ds.
		\end{equation*}
		Letting $n\to\infty$ taking into account the uniform convergence of $x_{ \lambda_n}$, the weak convergence of the sequences $v_{\lambda_n}$, $w_{\lambda_n}$ and $F$ is locally Lipschitz we have
		\begin{equation*}
			\forall t\in [0,\beta]:F(x(t))= F(x_0)-\int_{0}^t \left( v_1(s)+v_2(s)\right)ds.
		\end{equation*}
		\noindent\textbf{Claim 5:} The functions  $x$, $v $ and $ w$ satisfy that a.e. $v(t) \in \partial \varphi_1(x(s))$ and $w(t) \in {\partial } \varphi_2(x(s))$.\\
\noindent\emph{Proof of Claim 5:} Indeed, on the one hand, using Lemma \ref{env_e}, we get that   for all $n\in\N$ and $t\in [0,\beta]$, $v_{\lambda_n}(t)\in  {\partial} \varphi_1(P_{\lambda_n}\bar{\varphi}_1(x_{\lambda_n}(t)))$. From \cite[Lemma 1.9]{MR2252239} it follows that $v(t)\in \partial \varphi_1(x(t))$ a.e. on $[0,\beta]$. On the other hand, since $\{\nabla \varphi_2^{\lambda_n}(x_{\lambda_n}(t)): n\in \N\}$ is bounded for all $t\in [0,\beta]$, we can use \cite[Proposition 6.6.33]{MR2527754}, so  
        \begin{equation*}
            w(t)\in \ov{\operatorname{co}}(\limsup_{n\to \infty} \{\nabla \varphi_2^{\lambda_n}(x_{\lambda_n}(t))\}), \ \text{a.e. } t\in [0,\beta].
        \end{equation*}
       Now, for a fixed $t\in [0,\beta]$, we observe that 
       \begin{equation*}
           \ov{\operatorname{co}}(\limsup_{n\to \infty} \{\nabla \varphi_2^{\lambda_n}(x_{\lambda_n}(t))\})\subset \ov{\operatorname{co}} \ G(x(t)),
       \end{equation*}
        where $G(y) := \{\lim_{n\to\infty} \nabla \varphi_2^{\lambda_n}(y_{\lambda_n}):y_{\lambda_n}\to y \text{ and }\lambda_n \to 0\}$. By \cite[Theorem 9.67]{MR1491362}, we have $\operatorname{co}G(x(t)) = \operatorname{co}\partial \varphi_2(x(t))$ it follows that 
        \begin{equation*}
            \ov{\operatorname{co}}(\limsup_{k\to \infty} \{\nabla \varphi_2^{\lambda_n}(x_{\lambda_n}(t))\})\subset \ov{\operatorname{co}} \ G(x(t))=\ov{\operatorname{co}}\ \partial \varphi_2(x(t)) = \partial \varphi_2(x(t)).
        \end{equation*}
        Then, we conclude a.e. $w(t)\in\partial \varphi_2(x(t))$.\\
   \noindent\textbf{Claim 6:} The function  $\varphi \circ x$ is absolutely continuous.\\   
\noindent 	\emph{Proof of Claim 6:} In fact, first we note that $\varphi_2\circ x$ is   absolutely continuous since $\varphi_2$ is locally Lipschitz. Hence, we focus on  $\varphi_1\circ x$. We can see that $x(t)\in U=\mathbb{B}_{r_0}[x_0]$ for all $t\in [0,\beta]$, which implies that $P_\lambda\bar{\varphi}_1(x(t)) \in \mathbb{B}_{s_0}(x_0)$ for all $\lambda < \lambda_0$. Then, using  \eqref{hypo_plr} we  yield for every $\lambda\in ]0,\lambda_0[$ and a.e. $s\in [0,\beta]$
\begin{equation*}
    \begin{aligned}
        \langle \nabla e_\lambda\bar{\varphi}_1(x(s))+v (s),-\nabla e_\lambda\bar{\varphi}_1(x(s)) \rangle
            =& \ \frac{1}{\lambda}\langle \nabla e_\lambda\bar{\varphi}_1(x(s))-(-v (s)),P_\lambda\bar{\varphi}_1(x(s))-x(s) \rangle\\ 
   \geq & \ \frac{-\hat{c}}{\lambda}(1+\|\nabla e_\lambda\bar{\varphi}_1(x(s))\|+\|v (s)\|)\|P_\lambda\bar{\varphi}_1(x(s))-x(s)\|^2\\
   =& -\hat{c}(1+\|\nabla e_\lambda\bar{\varphi}_1(x(s))\|+\|v (s)\|)\lambda\|\nabla e_\lambda\bar{\varphi}_1(x(s))\|^2.
        \end{aligned}
\end{equation*}
	Now, by definition of Moreau envelope and \ref{FixH5}, we have
		\begin{equation*}
			\begin{aligned}
				\lambda\|\nabla e_\lambda\bar{\varphi}_1(x(s))\|^2 \leq & \ \varphi_1(x(s))-\varphi_1(P_\lambda\bar{\varphi}_1(x(s)))
                =\ \varphi(x(s)) + \varphi_2(x(s)) - \varphi_1(P_\lambda\bar{\varphi}_1(x(s)))\\
				\leq & \ \varphi(x_0)+\varphi_2(x(s)) +\alpha(\|P_\lambda\bar{\varphi}_1(x(s))\| + 1)
    -\varphi_2(P_\lambda\bar{\varphi}_1(x(s)))\\
                \leq & \ \varphi(x_0) + L\|x(s)-P_\lambda\bar{\varphi}_1(x(s))\|+\alpha(\|P_\lambda\bar{\varphi}_1(x(s))\| + 1),
			\end{aligned}
		\end{equation*}
  as we know $P_\lambda\bar{\varphi}_1(x(s))\to x(s)$ for all $s\in [0,\beta]$ and for all $\lambda\in ]0,\lambda_0[$, $P_\lambda \bar{\varphi}_1$ is $k_0$-Lipschitz according Lemma \ref{env_b} on $\mathbb{B}_{r_0}(x_0)$, we conclude the right hand side of the last inequality is bounded uniformly, i.e. we can say there is $\varrho>0$ such that 
  \begin{equation*}
      \sup_{\lambda\in ]0,\lambda_0[}\sup_{s\in [0,\beta]}\lambda \|\nabla e_\lambda\bar{\varphi}_1(x(s))\|^2\leq \varrho,
  \end{equation*}
  hence we have for a.e. $s\in [0,\beta]$ and $\lambda\in ]0,\lambda_0[$
		\begin{equation*}
			\|\nabla e_\lambda\bar{\varphi}_1(x(s))\|^2 - (\|v (s)\| + \hat{c}\varrho)\|\nabla e_\lambda\bar{\varphi}_1(x(s))\| - \hat{c}\varrho(\|v (s)\|+1)\leq 0.
		\end{equation*}
		From Lemma \ref{numbers_lemma}, it follows that there is a constant $\zeta>0$ independent of $s\in [0,\beta]$ and $\lambda\in ]0,\lambda_0[$
		\begin{equation*}
			\|\nabla e_\lambda\bar{\varphi}_1(x(s))\|\leq \zeta(\|v (s)\|+1), \text{ a.e. } s\in [0,\beta].
		\end{equation*}
  Since $v \in L^2([0,\beta];\R^d)$, we have $(\nabla e_\lambda\bar{\varphi}_1(x(s)))_{\lambda}$ is a bounded family in $L^2([0,\beta];\R^d)$, then we can pass to a weakly convergent subsequence, let us say, $\nabla e_{\lambda_k}\bar{\varphi}_1(x(s))\rightharpoonup  w$ in $L^2([0,\beta];\R^d)$ for some $\lambda_k\searrow 0$. Then we have for all $s,t\in [0,\beta]$
		\begin{equation*}
			e_{\lambda_k}\bar{\varphi}_1(x(t))-e_{\lambda_k}\bar{\varphi}_1(x(s)) = \int_s^t\langle \nabla e_{\lambda_k}\bar{\varphi}_1(x(\tau)),\dot{x}(\tau)\rangle d\tau,
		\end{equation*}
		so using that $\dot{x}\in L^2([0,\beta];\R^d)$, the weak convergence and Lemma \ref{env_g}, we get
		\begin{equation*}
			\varphi_1(x(t))-\varphi_1(x(s)) = \int_s^t \langle w(\tau),\dot{x}(\tau)\rangle d\tau,
		\end{equation*} 
		it proves that $\varphi_1\circ x$ is absolutely continuous on $[0,\beta]$.
		
			\noindent\textbf{Claim 7:} The function $x$ satisfies \eqref{integral_prop_lemma} on $[0,\beta]$.\\
\noindent 		\emph{Proof of Claim 7:} For all $n \in \mathbb{N}$ and all $s,t \in [0,\beta] $ with $s < t$, the following inequality holds
    \begin{equation*}
        \begin{aligned}
			\rho\int_s^t \|\dot{x}_{\lambda_n}(\tau)\|^2 d\tau\leq & \ \varphi_{\lambda_n}(x_{\lambda_n}(s))-\varphi_{\lambda_n}(x_{\lambda_n}(t)) \\= & \  \varphi_1   (P_{\lambda_n}\bar{\varphi}_1(x_{\lambda_n}(s))   ) +   \frac{1}{2\lambda_n}  \| x_{\lambda_n}(s) - P_{\lambda_n}\bar{\varphi}_1(x_{\lambda_n}(s)) \|^2\\&+ \varphi_2^{\lambda_n}(x_{\lambda_n}(s) )  - \varphi_{\lambda_n}(x_{\lambda_n}(t)) \\
			=& \  \varphi_1   (P_{\lambda_n}\bar{\varphi}_1(x_{\lambda_n}(s))   ) +   \frac{\lambda_n}{2}  \|\nabla e_{\lambda_n}\bar{\varphi}_1(x_{\lambda_n}(s))  \|^2\\
			&+ \varphi_2^{\lambda_n}(x_{\lambda_n}(s) )  - \varphi_{\lambda_n}(x_{\lambda_n}(t)).
		\end{aligned}
    \end{equation*}
		Since $ P_{\lambda_n}\bar{\varphi}_1(x_{\lambda_n}(\tau))) \to x(\tau) $ for all $\tau\in [0,\beta]$, we conclude by Lemma \ref{env_g} applied to $\varphi_1$, the uniform convergence of $\varphi_2^{\lambda_n}$ to $\varphi_2$ and the weak convergence of $(\dot{x}_{\lambda_n})$ that
			\begin{equation}
				\begin{aligned}\label{ineq_intets}
						\rho\int_s^t \|\dot{x}(\tau)\|^2 d\tau\leq&  \liminf_{n\to \infty} \left(  \varphi_1   (P_{\lambda_n}\bar{\varphi}_1(x_{\lambda_n}(s))   ) +  \frac{\lambda_n}{2}  \|\nabla e_{\lambda_n}\bar{\varphi}_1 (x_{\lambda_n}(s))  \|^2  \right)\\
					&	+  \varphi_2(x(s) )  - \varphi (x (t)).
				\end{aligned}
		\end{equation}
		Now, using Claim 2, we have the sequence  $(\nabla e_{\lambda_n}\bar{\varphi}_1(x_{\lambda_n}(\tau)))$ is  bounded   in $L^2([0,\beta], \R^d)$ for all $\tau\in [0,\beta]$. Then, by Fatou's Lemma, we have that 
		\begin{equation*}
			\int_0^\beta \liminf_{n \to \infty} \| \nabla e_{\lambda_n}\bar{\varphi}_1(x_{\lambda_n}(\tau)) \| d\tau  \leq \liminf_{n \to \infty} \int_0^\beta \| \nabla e_{\lambda_n}\bar{\varphi}_1(x_{\lambda_n}(\tau))\| d\tau.
		\end{equation*}
		It implies that there exists a measurable set $B \subset [0,\beta]$ of full measure such that 
		\begin{equation*}
			 \liminf_{n \to \infty} \| \nabla e_{\lambda_n}\bar{\varphi}_1(x_{\lambda_n}(t)) \|<\infty \text{ for all } t \in B.
		\end{equation*}
   Now, for each $\tau \in B$, we can take a bounded subsequence (which depends on $\tau$), let  us say,  $\| \nabla e_{\lambda_{n_k}}\bar{\varphi}_1(x_{\lambda_{n_k}}(\tau)) \| \leq M_\tau $ for all $k \in \mathbb{N}$. Hence, using the definition of plr function, we get that
   \begin{equation*}
       \begin{aligned}
   	\varphi_1(x(\tau)) \geq & \ \varphi_1 (P_{\lambda_{n_k}}\bar{\varphi}_1(x_{\lambda_{n_k}}(\tau))   )  +\langle  \nabla e_{\lambda_{n_k}}\bar{\varphi}_1(x_{\lambda_{n_k}}(\tau)) , x(\tau)-P_{\lambda_n}\bar{\varphi}_1(x_{\lambda_n}(\tau))    \rangle\\ &-c(1+\| \nabla e_{\lambda_{n_k}}\bar{\varphi}_1(x_{\lambda_{n_k}}(\tau)) \|)\|x(\tau)-P_{\lambda_n}\bar{\varphi}_1(x_{\lambda_n}(\tau))\|^2  \\
   	\geq & \ \varphi_1(P_{\lambda_{n_k}}\bar{\varphi}_1(x_{\lambda_{n_k}}(\tau)))    - M_\tau \|x(\tau)-P_{\lambda_n}\bar{\varphi}_1(x_{\lambda_n}(\tau))\|  \\
   	&- c(1 + M_\tau)\|x(\tau)-P_{\lambda_n}\bar{\varphi}_1(x_{\lambda_n}(\tau))\|^2.
   \end{aligned}
   \end{equation*}
   Hence, 
   \begin{equation*}
   	\liminf_{k\to \infty}   \varphi_1   (P_{\lambda_{n_k}}\bar{\varphi}_1(x_{\lambda_{n_k}}(\tau))) \leq \varphi_1(x(\tau)).
   \end{equation*}
   Therefore, for all $\tau\in B$
   \begin{equation*}
       \begin{aligned}
  	&\liminf_{n\to \infty} (  \varphi_1   (P_{\lambda_n}\bar{\varphi}_1(x_{\lambda_n}(\tau))   ) +  \frac{\lambda_n}{2}  \|\nabla e_{\lambda_{n}}\bar{\varphi}_1( x_{\lambda_n}(\tau))  \|^2  )\\
  	 \leq & \liminf_{k\to \infty} (  \varphi_1   (P_{\lambda_{n_k}}\bar{\varphi}_1(x_{\lambda_{n_k}}(\tau))   ) +  \frac{\lambda_{n_k}}{2}  \|\nabla e_{\lambda_{n_k}}\bar{\varphi}_1 (x_{\lambda_{n_k}}(\tau))  \|^2  )\\ 
  	 \leq &\liminf_{k\to \infty} (  \varphi_1   (P_{\lambda_{n_k}}\bar{\varphi}_1(x_{\lambda_{n_k}}(\tau))   ) +  \frac{\lambda_{n_k}}{2}  M_\tau^2  )\\
  	  =& \liminf_{k\to \infty} (  \varphi_1   (P_{\lambda_{n_k}}\bar{\varphi}_1(x_{\lambda_{n_k}}(\tau))   ) ) \leq \varphi_1(x(\tau)).
  	\end{aligned}
   \end{equation*}
   Replacing the above inequality in \eqref{ineq_intets}, we get that for all $t\in [0,\beta]$ and all $s\in B$  with $s < t$,
   	\begin{equation*}
   	\rho\int_s^t \|\dot{x}(\tau)\|^2 d\tau \leq   \varphi_1(x(s) ) +  \varphi_2(x(s) )  - \varphi (x (t)) = \varphi (x (s))-\varphi (x (t)).
   \end{equation*}
  From Claim 5, the map  $t \mapsto \varphi \circ x(t)$ is continuous (in fact, absolutely continuous). Therefore,  the last equality holds for all $s <t$, which concludes the proof of the claim.\\
			\noindent\textbf{Claim 8:} The statement of the theorem holds.\\
	 	 \noindent \emph{Proof of Claim 8:} By the previous claims, we have that 	$\mathcal{K}_{\varphi_1, \varphi_2, F,\Omega}(x_0)$ is nonempty. Hence, Proposition \ref{Prop_Maximal} implies the existence of a maximal element $(x,v,w,\tau) \in \mathcal{K}_{\varphi_1, \varphi_2, F,\Omega}(x_0)$. Let us show that $\tau = \infty$. Indeed, if $\tau<\infty$, by using \eqref{integral_prop_lemma} one has for all $t<\tau$
    \begin{equation*}
        \begin{aligned}
    \|x(t)-x_0\|^2 &\leq \left(\int_0^t\|\dot{x}(s)\|ds\right)^2\leq t\int_0^t\|\dot{x}(s)\|^2ds\\
    &\leq \tau(\varphi(x_0)-\varphi(x(t)))\leq \tau(\varphi(x_0) +\alpha(\|x(t)\|+1)),
        \end{aligned}
    \end{equation*}
where it has been used the condition \ref{FixH5}. Since the above inequality is valid for all $t\in [0,\tau[$, we conclude that $r:=\sup_{t<\tau}\|x(t)\|<\infty$, then if $t,s\in [0,\tau[$
\begin{equation*}
    \|x(t)-x(s)\|\leq \sqrt{|t-s|}\left(\int_s^t\|\dot{x}(t')\|^2dt'\right)^{1/2}\leq \tau(\varphi(x_0) +\alpha(r+1))\sqrt{|t-s|},
\end{equation*}
it follows that $\lim_{t\nearrow \tau}x(t)$ exists. We can argue in a similar way as Claim 1, by taking
		  \begin{equation*}
		 	\hat{x}_0 := \lim_{t\nearrow \tau}x(t) \text{ and  }   \hat{y}_0 := F(x_0) + \int_0^{\tau } \left( v(s) + w(s)  \right)     ds.
		 \end{equation*}
		 Using  \eqref{integral_prop_lemma} we get that $	\hat{x}_0  \in \Omega$, and $ \hat{y}_0 = F(\hat{x}_0)$. Hence, applying the previous construction, we can extend $x,v,w$ to some interval $[0,\hat{\tau} ]$ with $\tau < \hat{\tau}$ satisfying all the properties of the functions in $\mathcal{K}_{\varphi_1, \varphi_2, F,\Omega}(x_0)$, which contradicts the maximality, then it proves the claim.\\
\noindent\textbf{Claim 9:} If either \ref{Uniq1} or \ref{Uniq2} holds, the system \eqref{general_system} has an unique solution.\\
   \noindent \emph{Proof of Claim 9:} Take two trajectories satisfying \eqref{general_system}, let us say, $x_1,x_2\colon \R_+\to\R^d$ which are absolutely continuous. Since $F$ is locally Lipschitz, from Proposition \ref{equivalence_prop}, there are measurable functions $v_i\colon \R_+\to \R^d$ such that for a.e. $t\geq 0$, $v_i(t)\in\partial\varphi(x_i(t))$ and 
		\begin{equation}\label{inclusion-asd}
			F(x_i(t))=F(x_0) -\int_0^tv_i(s)ds \quad \textrm{ for }i\in \{1,2\}.
		\end{equation}
		Take $\mathscr{G}\subset \R_+$ the set of points where \eqref{inclusion-asd} holds, we know that $\R_+\setminus \mathscr{G}$ is a null set. Proceeding by contradiction, let us assume that $x_1\neq x_2$. We set $T:=\inf\{s\in \mathscr{G}:x_1(s)\neq x_2(s)\}$. It is clear that $T\geq 0$, and by our assumption we have $T<\infty$. Since $\mathscr{G}$ is dense on $\R_+$ we have $x_1(t) = x_2(t)$ for all $t\leq T$. Define $\bar{x}:=x_1(T)=x_2(T)\in \Omega$.\\
      Suppose that $(a)$ holds, since $\varphi$ is plr at $\bar{x}$, there is $\delta>0$ such that \eqref{hypo_plr} is satisfied on $\mathbb{B}_{\delta}(\bar{x})$ for some $c\geq 0$. We have there is $\varepsilon>0$ such that for all $t\in [T,T+\varepsilon]$, $x_i(t)\in \mathbb{B}_{\delta}(\bar{x})$. Since $DF$ is $\rho$-lower-definite, we take a positive definite matrix $H$ such that $F = H^2$. Thus, for a.e. on $t\in [T,T+\varepsilon]$ using \eqref{hypo_plr} one has
      \begin{equation*}
          \begin{aligned}
            -\frac{d}{dt}(\|Hx_1(t)-Hx_2(t)\|^2)=& \ 2\langle -F\dot{x}_1(t)+F\dot{x}_2(t),x_1(t)-x_2(t)\rangle\\
   \geq & \ -2c(1+\|F\dot{x}_1(t)\|+\|F\dot{x}_2(t)\|)\|x_1(t)-x_2(t)\|^2.
		\end{aligned}
      \end{equation*}
Since $F$ is linear, $L:=\sup_{\|v\|\leq 1}\|F(v)\|<\infty$, then $\|F\dot{x}_i(t)\|\leq L\|\dot{x}_i(t)\|$ for $i\in \{1,2\}$. Integrating the last inequality, we conclude that for all $t\in [T,T+\varepsilon]$
\begin{equation*}
    \begin{aligned}
\|x_1(t)-x_2(t)\|^2
   \leq & \frac{2c}{\lambda_{\text{min}}(H)^2} \int_T^{t} (1+L(\|\dot{x}_1(s)\|+\|\dot{x}_2(s)\|))\|x_1(s)-x_2(s)\|^2 ds,
		\end{aligned}
\end{equation*}
where $\lambda_{\text{min}}(H)$ denotes the smallest eigenvalue of $H$. For $i\in\{1,2\}$ we have $\dot{x}_i\mathbbm{1}_{[T,T+\varepsilon]}$ belongs to $ L^1(\R_+;\R^d)$, by using Gronwall's Lemma, we conclude that $x_1=x_2$ on $[T,T+\varepsilon]$, but it contradicts the definition of $T$. \\
Now, suppose $(b)$ holds. From \eqref{inclusion-asd}, we have for all $t\geq T$, $F(x_i(t)) = F(\bar x) - \int_T^t \nabla\varphi(x_i(s))ds$. Since $F$ is strongly metrically regular at $(\bar x, F(\bar x))$, we can apply Lemma \ref{lemma_local_solution} in order find $\delta>0$ such that the system \begin{equation*}
    F(w(t)) = F(\bar x) - \int_0^t \nabla\varphi(w(t)) ds, w(0) = \bar x
\end{equation*}
has an unique absolutely continuous solution  on $[0,\delta[$. Since $x_1(T+\cdot)$ and $x_2(T+\cdot)$ solve the system, we have $x_1(t) = x_2(t)$ on $[T,T+\delta[$, which contradicts the definition of $T$.\\
The conclusion in both cases is that $T=\infty$; thus, we have uniqueness. \qed

\vspace{-0.2cm}
\section{Proof of the Asymptotic Results}

 \subsection{Proof of Theorem \ref{asym_behavior}}\label{theorem_asym_behavior}
 $(a)$: It is well-known that $\omega(x) = \bigcap_{t\geq 0}\overline{x([t,\infty[)}$. Hence, for any continuous and bounded trajectory $x(\cdot)$, the set $\omega(x)$ is nonempty, compact, and connected. \\
 $(b)$: We observe that $x(\cdot)$ is an energetic solution and, by virtue of \ref{FixH1}, the trajectory $x(\cdot)$ remains in $\Omega$. Let $x^{\ast}\in \omega(x)$. By definition, there exists a sequence $t_k\to \infty$ such that $x(t_k)\to x^{\ast}$. Since $x(\cdot)$ is an energetic solution, there exists $\rho>0$ such that
 $$
 \rho \int_0^{t_k}\Vert \dot{x}(\tau)\Vert^2 d\tau \leq \varphi(x_0)-\varphi(x(t_k)) \quad \textrm{ for all } k\in \mathbb{N},
 $$
 which, by the lower semicontinuity of $\varphi$ (see assumption \ref{FixH2}), implies that $\dot{x}\in L^2(\mathbb{R}_+;\mathbb{R}^d)$. Hence, we can choose a subsequence $(t_{\nu(k)})_k$ such that $t_{\nu(k+1)}-t_{\nu(k)}\geq 1$ for all $k\in \mathbb{N}$.  We set $\mathcal{I}_{k}:=]t_{\nu(k)}-\frac{1}{k},t_{\nu(k)}+\frac{1}{k}[$.  Denote by $\mathcal{N}\subset \mathbb{R}_+$ a null set such that for all $t\in \mathbb{R}_+\setminus \mathcal{N}$,  $\dot{x}(t)$ exists and $-(\partial \varphi_1(x(t))+\partial\varphi_2(x(t)))\cap DF(x(t))(\dot{x}(t))\neq\emptyset$.\\
 \textbf{Claim 1:} For all $n\in\N$, there is $k\geq n$ and $s\in \mathcal{I}_{k}\setminus \mathcal{N}$ such that $\|\dot{x}(s)\|<\frac{1}{\sqrt{n}}$.\\
  \emph{Proof of Claim 1:} By contradiction, suppose there exists $n\in\N$ such that for all $k\geq n$ and $s \in \mathcal{I}_{k}\setminus \mathcal{N}$, we have $\|\dot{x}(s)\|\geq \frac{1}{\sqrt{n}}$. Then,  for all $k\geq n$, we have
		\begin{equation*}
			\int_{\mathcal{I}_{k}}\|\dot{x}(s)\|^2 ds\geq \frac{2}{kn}.
		\end{equation*}
Hence, by using that $t_{\nu(k+1)}-t_{\nu(k)}\geq 1$ for all $k\in \mathbb{N}$, we obtain
		\begin{equation*}
			\int_{0}^{\infty}\|\dot{x}(s)\|^2ds\geq \sum_{k=n}^{\infty} \int_{\mathcal{I}_{k}} \|\dot{x}(s)\|^2ds\geq \frac{2}{n}\sum_{k=n}^{\infty}\frac{1}{k} = \infty,
		\end{equation*}
  which contradicts the fact that $\dot{x}\in L^2(\R_+;\R^d)$.  Hence, the claim is proved.\\
   From the above claim, we have established  the existence of a strictly increasing function $\varsigma\colon \N\to\N$ and a sequence $(s_n)$ such that, for all $n\in \mathbb{N}$,
  $$
  \|\dot{x}(s_n)\|<\frac{1}{\sqrt{n}} \textrm{ and } s_n\in \mathcal{I}_{\varsigma(n)}\setminus \mathcal{N}.
  $$
\textbf{Claim 2:} $x(s_n)\to x^\ast$  as $n\to \infty$.\\
\emph{Proof of Claim 2:} Indeed, by virtue of Hölder's inequality, for all $n\in \mathbb{N}$, we have 
\begin{equation*}
				\|x(s_n)-x(t_{\nu(\varsigma(n))})\| \leq  \int_{\mathcal{I}_{\varsigma(n)}}\|\dot{x}(s)\|ds\leq \sqrt{\frac{2}{\varsigma(n)}}\left(\int_{0}^{\infty}\|\dot{x}(s)\|^2ds\right)^{1/2}.
		\end{equation*}
  Hence, the claim follows from the fact that $\dot{x}\in L^2(\mathbb{R}_+;\mathbb{R}^d)$ and that $\varsigma(n)\to \infty$.\\
  \textbf{Claim 3:} $\operatorname{dist}(0;\partial \varphi_1 (x(s_n)) + \partial \varphi_2 (x(s_n))) \to 0$ as $n\to \infty$.\\
  \emph{Proof of Claim 3:} For all $n\in \mathbb{N}$, since $s_n \in \mathbb{R}_+ \setminus \mathcal{N}$, there exist $v_n\in \partial \varphi_1(x(s_n))$ and $w_n\in \partial \varphi_2(x(s_n))$ such that $-(v_n+w_n)\in DF(x(s_n))(\dot{x}(s_n))$. By virtue of assumptions \ref{FixH2} and \ref{FixH3}, the functions $F$ and $\varphi_2$ are Lipschitz around $x^{\ast}$ with constant $L>0$. Hence,  due to Claim 2, for all sufficiently large $n\in \mathbb{N}$, 
  \begin{equation}\label{eqnew01}
    \operatorname{dist}(0;\partial \varphi_1 (x(s_n)) + \partial \varphi_2 (x(s_n))) \leq \|v_n+w_n\|\leq L\|\dot{x}(s_n)\| \textrm{ and }\|w_n\|\leq L,
  \end{equation}  
which implies the claim, since $\dot{x}(s_n) \to 0$ as $n\to \infty$.\\
\textbf{Claim 4:} $\varphi(x(s_n)) \to \varphi(x^\ast)$ as $n\to \infty$ and $x^\ast\in \mathcal{S}$.\\
\emph{Proof of Claim 4}: Indeed, due to \ref{FixH2}, the function $\varphi_2$ is continuous at $x^{\ast}$. Hence, it is sufficient to prove that $\varphi_1(x(s_n))\to \varphi_1(x^{\ast})$ as $n\to \infty$.  To establish this convergence, we observe that, by virtue of \ref{FixH2},  the subdifferential of $\varphi_2$ is bounded around $x^\ast$. Thus, by  \eqref{eqnew01}, $\sup_{n\in \mathbb{N}}\operatorname{dist}(0,\partial \varphi_1(x(s_n)))<\infty$. Therefore, the convergence follows from   \cite[Corollary 1.7]{MR2252239}. Next, take sequences $v_{n_k}\to v$ and $w_{n_k}\to w$. Since $\dot{x}(s_n)\to 0$ as $n\to \infty$,  we have that $v+w=0$. Moreover, by virtue of \ref{FixH2}, $v\in \partial \varphi_1(x^{\ast})$ and $w\in \partial \varphi_2(x^{\ast})$, which proves that $x^{\ast}\in \mathcal{S}$.\\
\textbf{Claim 5}: If $x^{\ast}$ is an isolated point of $\omega(x)$, then $x(t)\to x^{\ast}$ as $t\to \infty$.\\
\emph{Proof of Claim 5:} Assume that $x^{\ast}$ is an isolated point of $\omega(x)$.  Then there exists $\varepsilon>0$ such that $\mathbb{B}_{\varepsilon}(x^\ast)\cap \omega(x) = \{x^\ast\}$. According to assertion $(a)$ and \cite[Lemma 1.35]{MR3616647}, to prove that $x(t)\to x^{\ast}$, it is sufficient  to prove that $x(\cdot)$ is bounded. \\
Assume, by contradiction, that $x(\cdot)$ is unbounded. Let $t_k\to \infty$ be a strictly increasing sequence such that $x(t_k)\to x^{\ast}$. Then, there exists $N_0\in \mathbb{N}$ such that for all $k\geq N_0$, we have $\Vert x(t_k)-x^{\ast}\Vert <\varepsilon/2$. Since we are assuming that $x(\cdot)$ is unbounded, there exists a strictly increasing $\theta\colon \mathbb{N}\to \mathbb{N}$ with $\theta(1)\geq N_0$ such that, for all $k\in \mathbb{N}$, there exists $t\in [t_{\theta(k)},t_{\theta(k+1)}]$ satisfying $\Vert x(t)-x^{\ast}\Vert \geq \varepsilon/2$. Next, define 
\begin{equation*}
    \tau_k := \inf \{ t \geq t_{\theta(k)} : \|x(t) - x^\ast\| \geq \varepsilon / 2 \},
\end{equation*}
which, by the definition of $\theta$, is well defined and satisfies $\tau_k \in ]t_{\theta(k)}, t_{\theta(k+1)}[$ for all $k \in \mathbb{N}$. Furthermore, by the continuity of $x(\cdot)$, we have
\begin{equation}\label{eqn_convergence}
 \|x(\tau_k) - x^\ast\| \geq \varepsilon / 2 \quad \textrm{ for all } k\in \mathbb{N}.
\end{equation}
It is also clear that for all $t \in [t_{\theta(k)}, \tau_k[$, we have $\|x(t) - x^\ast\| < \varepsilon / 2$. Thus, for all $k \in \mathbb{N}$, there exists $\delta_k \in ]0, 1]$ such that $\|x(\tau_k - \delta_k) - x^\ast\| < \varepsilon / 2$. On the other hand, since $\omega(x) \cap \mathbb{B}_{\varepsilon}(x^\ast) = \{x^\ast\}$, up to a subsequence, we have $x(\tau_k - \delta_k) \to  x^\ast$ as $k\to \infty$. However, for all $k \in \mathbb{N}$,
\begin{equation*}
    \begin{aligned}
        \|x(\tau_k) - x(\tau_k - \delta_k)\| &\leq \int_{\tau_k - \delta_k}^{\tau_k} \|\dot{x}(s)\| \, ds 
        \leq \sqrt{\delta_k} \left( \int_{\tau_k - \delta_k}^{\tau_k} \|\dot{x}(s)\|^2 \, ds \right)^{1/2} \\
        &\leq \left( \int_{\tau_k - 1}^{\infty} \|\dot{x}(s)\|^2 \, ds \right)^{1/2}.
    \end{aligned}
\end{equation*}
Since the right-hand side of the above inequality tends to 0, we conclude that $x(\tau_k) \to x^\ast$ as $k\to \infty$, which contradicts \eqref{eqn_convergence}. Hence, $x(\cdot)$ is bounded, and the proof is finished.\qed

\vspace{-0.4cm}

\subsection{Proof of Theorem \ref{rateundermetric_reg}}\label{Proofofrateundermetric_reg}  For simplicity, define $\mathscr{D}(x):=\partial\varphi_1(x)+\partial\varphi_2(x)$. Let $v(\cdot)$ be a measurable selection of $t\tto \mathscr{D}(x(t))$ such that 
$$-v(t ) \in D F(x(t)) (\dot{x}(t)) \quad \textrm{ for all  } t\in \mathbb{R}_+\setminus \mathcal{N},
$$
where $\mathcal{N}$ is a null set on $ \mathbb{R}_+$. Since, $\mathscr{D}$ is strongly metrically subregular at $(x^\ast,0)$, there exists $\eta >0$ such that 
	\begin{equation}\label{msubreg}
		\|y-x^\ast\|\leq \kappa \cdot \operatorname{dist}(0;\mathscr{D}(y)),\quad \textrm{ for all } y\in \mathbb{B}_{\eta}[x^\ast].
	\end{equation}
 \textbf{Claim 1:} $x(t) \to x^\ast$ as $t \to \infty$.\\
\emph{Proof of Claim 1:} Let $z\in \omega(x)\cap \mathbb{B}_{\eta}(x^\ast)$. According to Theorem \ref{asym_behavior}, we have $0\in \mathscr{D}(z)$. Moreover, from \eqref{msubreg}, it follows that $\|z-x^\ast\|\leq \kappa \cdot \operatorname{dist}(0;\mathscr{D}(z))= 0$, which implies that $x^{\ast}$ is an isolated accumulation point. Hence, by Theorem \ref{asym_behavior},  $x(t)\to x^{\ast}$ as $t\to \infty$.\\ 
   \noindent \textbf{Claim 2:} There exist constants $T,\mathfrak{b},\mathfrak{c}>0$ such that 
   $$
   \|v(t)\|\leq \mathfrak{b}\|\dot{x}(t)\| \textrm{ and } \|x(t)-x^\ast\|\leq \mathfrak{c}\|\dot{x}(t)\| \quad \textrm{ for a.e. } t\geq T.
   $$
   \emph{Proof of Claim 2:}   By Claim 1 and the local Lipschitzianity of $F$ around $x^{\ast}$ (see assumption \ref{FixH3}), there exist $\mathfrak{b} > 0$ and $T \geq 0$  such that 
\begin{equation*}
			\| x(t) - x^\ast\| \leq \eta  \textrm{ for all } t\geq T \, \textrm{ and } \, 
 \| v(t)  \|  \leq  \mathfrak{b} \| \dot{x} (t)\|  \text{ for a.e.  }t \geq T. 
\end{equation*} 	
Then, by using \eqref{msubreg} we have for a.e. $t\geq T$ 
\begin{equation*}
    \|x(t)-x^\ast\|\leq \kappa \|v(t)\|\leq \kappa\cdot \mathfrak{b}\|\dot{x}(t)\|,
\end{equation*}
and we conclude the claim.\\
  \noindent \textbf{Claim 3:} Inequalities \eqref{eq02}  and \eqref{eq_trivial}  hold.  \\
 \emph{Proof of Claim 3:} From Claim 2, it follows that for all $T\leq s\leq t$. 
  \begin{equation*}
      \int_s^t\|x(r)-x^\ast\|^2dr\leq \mathfrak{c}^2\int_s^t \|\dot{x}(r)\|^2dr\leq 
 \frac{\mathfrak{c}^2}{\rho}(\varphi(x(s))-\varphi(x(t))),
  \end{equation*}
thus establishing \eqref{eq02}. Next, setting $\alpha := \int_0^\infty\|\dot{x}(s)\|^2ds$, the above inequality implies that  
   \begin{equation*}
 (t-T)\min_{s\in [T,t]}\|x(s)-x^\ast\|^2\leq \mathfrak{c}^2\alpha \quad \textrm{ for all } t\geq T, 
   \end{equation*}
   which proves \eqref{eq_trivial}. The claim is proved.\\
\noindent \textbf{Claim 4:} If either  $(i)$ or $(ii)$ holds, then \eqref{eq01} is satisfied.  \\
\emph{Proof of Claim 4:}  Define $\mathscr{V}(t) := \varphi(x(t))-\varphi(x^\ast)$. Since the solution $x(\cdot)$ is energetic, the function $\mathscr{V}$ is a.e. differentiable. We will prove the existence of $\mathfrak{g},\tau>0$ such that $\mathscr{V}(s)\leq -\mathfrak{g}\dot{\mathscr{V}}(s)$ for a.e. $s\geq \tau$, which, by virtue of Grönwall's inequality, will imply \eqref{eq01}. \\
First case: assumption $(i)$ holds: Let $\sigma, \delta>0$ be such that $\varphi = \varphi_1$ is $\sigma$-weakly convex on $\mathbb{B}_{\delta}[x^\ast]$. Let $\tau_1\geq T$ be such that for all $t\geq \tau_1$, $x(t)\in\mathbb{B}_{\delta}[x^\ast]$. Next, since $\varphi=\varphi_1$ is $\sigma$-weakly convex and $v(s)\in\partial\varphi(x(s))$ for a.e. $s\geq \tau_1$, we can apply  \cite[Theorem 10.29]{MR4659163} to obtain that for a.e. $s\geq \tau_1$.
\begin{equation*}
    \begin{aligned}
        \mathscr{V}(s) &\leq \langle v(s),x(s)-x^\ast\rangle + \frac{\sigma}{2}\|x(s)-x^\ast\|^2\leq \|v(s)\| \|x(s)-x^\ast\| + \frac{\sigma}{2}\|x(s)-x^\ast\|^2\\
     &\leq \mathfrak{b} \mathfrak{c}\|\dot{x}(s)\|^2 + \frac{\sigma \mathfrak{c}^2}{2}\|\dot{x}(s)\|^2\leq -\frac{1}{\rho}(\mathfrak{b}\mathfrak{c} + \frac{\sigma\mathfrak{c}^2}{2})\mathscr{V}'(s),
    \end{aligned}
\end{equation*}
where we have used Claim 2 and the fact that $x(\cdot)$ is an energetic solution.\\ 
Finally, suppose that $\sigma<\kappa^{-1}$. Then, since $\partial\varphi_1$ is strongly metrically subregular at $(x^\ast,0)$ with modulus $\kappa$, by virtue of \cite[Corollary 3.3]{MR3331214}, there exist $\gamma',\delta'>0$ such that 
\begin{equation*}
    \varphi(x)\geq \varphi(x^\ast) + \gamma'\|x-x^\ast\|^2 \textrm{ for all } x\in \mathbb{B}_{\delta'}[x^\ast]. 
\end{equation*}
Hence, we can find $\tau_1'\geq \tau_1$ such that 
 \begin{equation*}
     \mathscr{V}(s)\geq \gamma'\|x(s)-x^\ast\|^2\quad \textrm{ for all } s\geq \tau_1',
 \end{equation*} 
 which proves \eqref{eqn-bound-varphi1}.\\
Second case: assumption $(ii)$ holds: Let $\sigma',\eta_2>0$ be such that $-\varphi = -\varphi_2$ is $\sigma'$-weakly convex on $\mathbb{B}_{\eta_2}[x^\ast]$. Consider $\tau_2\geq T$ be such that $x(t)\in \mathbb{B}_{\eta_2}(x^\ast)$ for all $t\geq \tau_2$. We observe that $0\in\partial\varphi(x^\ast) = -\partial(-\varphi)(x^\ast)$. Hence, by using \cite[Theorem 10.29]{MR4659163}, Claim 2 and that the $x(\cdot)$ is an energetic solution, we obtain that for a.e. $s\geq \tau_2$, 
 \begin{equation*}
    \mathscr{V}(s) = \varphi(x(s))-\varphi(x^\ast)\leq \frac{\sigma'}{2}\|x(s)-x^\ast\|^2\leq \frac{\sigma'\mathfrak{c}^2}{2}\|\dot{x}(s)\|^2\leq -\frac{\sigma'\mathfrak{c}^2}{2\rho }\mathscr{V}'(s),    
 \end{equation*}
 which ends the proof. \qed

\vspace{-0.2cm}
\subsection{Proof of Theorem \ref{thm-KL}}\label{ProofOf_thm-KL}  Since $x^\ast\in\omega(x)$, it follows from Theorem \ref{asym_behavior} that there exists an increasing sequence $(t_n)_{n\in\N}$ with $ t_n\to \infty$ such that $x(t_n)\to x^\ast$ and $\varphi(x(t_n)) \to \varphi(x^\ast)$ as $n\to\infty$. Since the solution $x(\cdot)$ is energetic, the map $t\mapsto \varphi(x(t))$ is nonincreasing. Hence,  $\varphi(x(t)) \to \varphi(x^\ast)$ as $t\to\infty$. Therefore, if $t_0\geq 0$ is such that $\varphi(x(t_0)) = \varphi(x^\ast)$, it is clear that $\varphi(x(t)) = \varphi(x^\ast)$ for all $t \geq t_0$. Consequently, from $\eqref{integral_prop}$, we obtain that $\dot{x}(t) =0$ and $x(t)=x^\ast$ for all $t \geq t_0$. Thus, without loss of generality,  we can assume that $\varphi(x(t)) >  \varphi(x^\ast) $ for  all $t \geq 0.$\\
    Let $\eta>0$ be such that  \eqref{KL} holds, and define $h(t) :=  \psi(\varphi(x(t))-\varphi(x^\ast))$.  Then, for a.e. $t\geq 0$, we have
		\begin{equation*}
				h'(t) = \psi'(\varphi(x(t))-\varphi(x^\ast)) \frac{d}{dt}(\varphi\circ x)(t).
		\end{equation*}
        By virtue of \ref{FixH3}, there is $L>0$ and $\varepsilon>0$ such that $F$ is $L$-Lipschitz on $\mathbb{B}_{\varepsilon}[x^\ast]$. Without loss of generality assume that $\varepsilon<\eta$. Since  $h(t)\to 0$ as $t\to \infty$, there exists $\hat t$ such that $h(\hat t)-h(t)< \frac{\rho}{3L}\varepsilon$ and $\varphi(x^\ast)<\varphi(x(t))<\varphi(x^\ast)+\eta$ for all $t\geq \hat t$. Additionally, since $x(t_n)\to x^\ast$,  we can assume that $\|x(\hat t)-x^\ast\|<\frac{\varepsilon}{3}$.\\ Set $T := \inf\{ t\geq \hat t: \|x(t)-x^\ast\|\geq \varepsilon \}$, and let $t\in [\hat t,T[$, where $x(\cdot)$ and $\varphi(x(\cdot))$ are differentiable and 
  \begin{equation*}
    -(\partial\varphi_1(x(t))+\partial\varphi_2(x(t)))\cap DF(x(t))(\dot{x}(t))\neq \emptyset.    
  \end{equation*}
  Let $v$ be an element of the above set. Then, given that $v\in DF(x(t))(\dot{x}(t))$ and $F$ is $L$-Lipschitz, we obtain that $\|v\|\leq L\|\dot{x}(t)\|$. Hence, from \eqref{KL}, it follows that 
  \begin{equation*}
          \|v\|\geq \operatorname{dist}(0;\partial\varphi_1(x(t)) + \partial\varphi_2(x(t)))>0.
  \end{equation*}
  Next, since $x(\cdot)$ is energetic, we obtain that $-\frac{d}{dt}(\varphi\circ x)(t)\geq \rho\|\dot{x}(t)\|^2$. Then, for all $t\in [\hat t, T[$,
		\begin{equation*}
			\begin{aligned}
				-h'(t) &= -\frac{d}{dt}(\varphi\circ x)(t)\cdot\psi'(\varphi(x(t))-\varphi(x^\ast))\\
				&\geq \frac{\rho\|\dot{x}(t)\|^2}{\text{dist}(0;\partial\varphi_1(x(t)) + \partial\varphi_2(x(t)))}
                \geq \frac{\rho\|\dot{x}(t)\|^2}{\|v\|}
                \geq \frac{\rho}{L}\|\dot{x}(t)\|,
			\end{aligned}
		\end{equation*}
        which implies that 
        $$\int_{\hat t}^{T}\|\dot{x}(s)\|ds\leq \frac{L}{\rho}(h(\hat t)-h(T))<\frac{\varepsilon}{3}.$$ 
        However, if $T<\infty$, it follows from the absolute continuity of $x(\cdot)$ that $\|x(\hat t)-x(T)\|<\frac{\varepsilon}{3}$. Then, $\|x(T)-x^\ast\|<\frac{2\varepsilon}{3}$, which contradicts the definition of $T$. Therefore, $T=\infty$ and the following inequality holds:
        \begin{equation*}
			\begin{aligned}
				\int_{0}^{\infty}\|\dot{x}(s)\|ds&\leq \left(\hat{t}\int_{0}^{\hat t}\|\dot{x}(s)\|^2ds\right)^{1/2}+\int_{\hat t}^{\infty}\|\dot{x}(s)\|ds\\
				&\leq  \sqrt{\frac{\hat{t}}{\rho}(\varphi(x_0)-\varphi(x(\hat t)))}+\int_{\hat t}^{\infty}\|\dot{x}(s)\|ds<\infty,
			\end{aligned}
		\end{equation*}
		which proves that $\dot{x}\in L^1(\R_+;\R^d)$. This fact implies that $(x(t))_{t\geq 0}$ has the Cauchy property, and given that $x^\ast\in\omega(x)$, we conclude that $x(t)\to x^\ast$ as $t\to\infty$. \qed

\vspace{-0.2cm}
\subsection{Proof of Corollary \ref{Coro-KL}}\label{Proof_of_Coro-KL}  Let us consider $\eta>0$ such that \eqref{KL} holds for all $x \in\mathcal{V}$, where   
     \begin{equation*}
        \mathcal{V} :=\mathbb{B}_{\eta}(x^\ast)\cap \varphi^{-1}(]\varphi(x^\ast),\varphi(x^\ast)+\eta[).    
     \end{equation*}
     Without loss of generality, we can assume that $F$ is $L$-Lipschitz continuous on $\mathbb{B}_{\eta}(x^\ast)$. Hence, by Theorem \ref{thm-KL}, we have that $x(t)\to x^\ast$ and $\varphi(x(t))\to \varphi(x^\ast)$ as $t\to \infty$. Therefore, there exists $T\geq 0$ such that $x(t)\in \mathcal{V}$ for all $t\geq T$. Let us define $\mathscr{V}(t) = \varphi(x(t))-\varphi(x^\ast)$. Using Lemma \ref{lemma_cotingent_lips}, the Lipschitz continuity of $F$, and \eqref{KL}, we have for a.e. $t\geq T$, 
     \begin{equation}\label{ineq-KL-00}
         1\leq \left\|\frac{d}{dt}(F\circ x)(t)\right\|\cdot M(1-\theta)\mathscr{V}(t)^{-\theta}\leq L M(1-\theta)\|\dot{x}(t)\| \mathscr{V}(t)^{-\theta}.
     \end{equation}
     Using that $x(\cdot)$ is energetic,  we obtain that, for all $t,s\geq T$ with $s\leq t$,
     \begin{equation*}
         \begin{aligned}
             \mathscr{V}(s)-\mathscr{V}(t)=\varphi(x(s))-\varphi(x(t)) &\geq \rho\int_s^t \|\dot{x}(\tau)\|^2d\tau
         \geq \frac{\rho}{(L M(1-\theta))^2}\int_s^t \mathscr{V}(\tau)^{2\theta}d\tau.
         \end{aligned}
     \end{equation*}
     Then, for a.e. $t\geq T$, we have
\begin{equation}\label{ineq-dif-KL}
        -\mathscr{V}'(t)\geq \rho \|\dot{x}(t)\|^2 \geq \alpha \mathscr{V}(t)^{2\theta},
     \end{equation}
     where  $\alpha := \frac{\rho}{(L M(1-\theta))^2}$.   Using \eqref{ineq-KL-00} and \eqref{ineq-dif-KL}, we obtain that, for a.e. $t\geq T$,
     \begin{equation*}
         -(\psi\circ \mathscr{V})'(t) = -\mathscr{V}'(t)\cdot \psi'(\mathscr{V}(t))\geq  \frac{\rho}{L}\|\dot{x}(t)\|.
     \end{equation*}
     This inequality implies that, for all $t,s\geq T$ with $t\geq s$,
     \begin{equation*}
         \|x(s)-x(t)\|\leq \int_s^t\|\dot{x}(\tau)\|d\tau\leq \frac{L}{\rho}(\psi(\mathscr{V}(s))-\psi(\mathscr{V}(t))).
     \end{equation*}
     Taking $t\to\infty$ and using that $\psi(0) = 0$ along with \eqref{ineq-KL-00}, we obtain that for all $s\geq T$,
     \begin{equation}\label{ineq-dif-KL-5}
         \|x(s)-x^\ast\|\leq \int_s^\infty \|\dot{x}(\tau)\|d\tau\leq \frac{L}{\rho}\psi(\mathscr{V}(s))=\frac{L M}{\rho}\mathscr{V}(s)^{1-\theta}\leq \mathfrak{w}\|\dot{x}(s)\|^{\frac{1-\theta}{\theta}},
     \end{equation}
     where $\mathfrak{w} := \frac{L M}{\rho}(L M (1-\theta))^{\frac{1-\theta}{\theta}}$.\\
     \textbf{Claim 1:} Assertion $(a)$ holds.\\
     \emph{Proof of Claim 1:} Indeed, assume that $\theta\in [0,1/2[$ and suppose that there exists  $T^\ast\geq T$ such that  $\mathscr{V}(T^\ast)=0$. In this case, we have that   $\mathscr{V}(t) = 0$ for all $t\geq T^\ast$ (recall that $\mathscr{V}$  is nonincreasing). Now, suppose, for the sake of contradiction that, 
 $\mathscr{V}(t)>0$ for all $ t\geq T^\ast$. Using \eqref{ineq-dif-KL},  we then have 
     \begin{equation*}
         -\alpha\geq \frac{\mathscr{V}'(t)}{\mathscr{V}(t)^{2\theta}} = \frac{1}{1-2\theta}\frac{d}{dt}(\mathscr{V}(t)^{1-2\theta}),  \quad \text{
     for a.e. }t\in [T,\infty[.
     \end{equation*}
    Hence, given that $\mathscr{V}^{1-2\theta}$ is nonincreasing, we have that  for all $t\in [T,\infty[$,
     \begin{equation}\label{theta-1/2-gronwall}
     \mathscr{V}(t)^{1-2\theta}\leq \mathscr{V}(T)^{1-2\theta}-\alpha(1-2\theta)(t-T).
     \end{equation}
Taking $t\to \infty$ in \eqref{theta-1/2-gronwall}, we arrive to a contradiction since $\mathscr{V}(t)>0$ for all $t\geq T$. This implies that $T^\ast<\infty$, so $(\varphi(x(t)))_{t\geq 0}$ converges in finite time to $\varphi(x^\ast)$. Finally, using \eqref{ineq-dif-KL-5}, we conclude that $x(t)$ also converges in finite time to $x^\ast$. \\
     \textbf{Claim 2:} If $\theta=1/2$, there exist $\nu_1,\nu_2>0$ such that 
     $$
     \mathscr{V}(x(t))\leq \nu_1\exp(-\alpha t)  \textrm{ and } \|x(t)-x^\ast\|\leq \nu_2\exp(-\alpha t/2) \textrm{ for all } t\geq T.
     $$
    \emph{Proof of Claim 2:} By virtue of \eqref{ineq-dif-KL}, we have $\mathscr{V}'(t)\leq -\alpha \mathscr{V}(t)$ for a.e. $t\geq T$. Thus, $\frac{d}{dt}(e^{\alpha t}\mathscr{V}(t))\leq 0$, which implies the existence of a constant $\nu_1>0$ such that 
    $$
    \mathscr{V}(t)\leq \nu_1\exp(-\alpha t) \textrm{ for all } t\geq T.
    $$
   Using \eqref{ineq-dif-KL-5}, we then have for all $s\geq T$, 
    \begin{equation*}
        \|x(s)-x^\ast\|\leq \frac{\kappa M}{\rho}\mathscr{V}(s)^{1/2}\leq \frac{\kappa M\sqrt{\nu_1}}{\rho}\exp{(-\alpha t/2)}.
    \end{equation*}
    \textbf{Claim 3:} If $\theta\in ]1/2,1[$, there exist $\mathscr{G}>0$ such that 
    $$
    \|x(s)-x^\ast\|\leq \mathscr{G} s^{\frac{1-2\theta}{1-\theta}} \textrm{ for all  }s\geq T.
    $$
    \emph{Proof of Claim 3:} Using \eqref{ineq-dif-KL-5}, we have for all $s\geq T$, 
    \begin{equation*}
        \int_s^\infty \|\dot{x}(\tau)\|d\tau\leq \mathfrak{w}\|\dot{x}(s)\|^{\frac{1-\theta}{\theta}}.    
    \end{equation*}
    Define $\phi(s) := \int_s^\infty\|\dot{x}(\tau)\|d\tau$. Note that $\phi$ is absolutely continuous. Then, there exists a constant $\mathscr{E}>0$ such that $\phi'(s)\leq \mathscr{E}\phi(s)^{\frac{\theta}{1-\theta}}$ for a.e. $s\geq T$. We then have 
    $$\frac{d}{ds}(\phi(s)^{\frac{1-2\theta}{1-\theta}})\leq \mathscr{E}\frac{2\theta-1}{1-\theta}.$$ 
    Integrating the above inequality, we obtain that for all $s\geq T$,
    \begin{equation*}
        \phi(s)^{\frac{1-2\theta}{1-\theta}}\leq \mathscr{E}\frac{2\theta-1}{1-\theta}(s-T) + \phi(T)^{\frac{1-2\theta}{1-\theta}}\leq \mathscr{E}\frac{2\theta-1}{1-\theta}s + \phi(T)^{\frac{1-2\theta}{1-\theta}}.
    \end{equation*}
    It follows that there exist constants $\mathscr{Y},\mathscr{F}>0$ such that
    \begin{equation*}
        \phi(s) \leq \frac{1}{(\mathscr{Y} + \mathscr{F}s)^{\frac{2\theta-1}{1-\theta}}}\leq (\mathscr{F}s)^{\frac{1-2\theta}{1-\theta}}.
    \end{equation*}
    We conclude the result using \eqref{ineq-dif-KL-5}. \\
    \textbf{Claim 4:} If $\theta\in ]1/2,1[$, there exists $\mathscr{C}>0$ s.t.  $\varphi(x(t))-\varphi(x^\ast)\leq \mathscr{C} s^{\frac{1}{1-2\theta}}$, for all $s\geq T$.\\
    \emph{Proof of Claim 4:} Observe that in this case, \eqref{theta-1/2-gronwall} remains valid as long as convergence in finite time does not occur (i.e., $\mathscr{V}(t)>0$ for all $t\geq T$). Since $\theta>1/2$, it follows that for all $s\geq T$,
    \begin{equation*}
        \mathscr{V}(s)\leq (V(T)^{1-2\theta} + \alpha(2\theta - 1)s)^{\frac{1}{1-2\theta}} \leq (\alpha(2\theta-1))^{\frac{1}{1-2\theta}} s^{\frac{1}{1-2\theta}}, 
    \end{equation*}
    which completes the proof. \qed

	\bibliographystyle{plain}
	\bibliography{references}
\end{document}